\documentclass[11pt]{amsart}
\usepackage{amsfonts}
\usepackage{geometry}                
\usepackage{graphicx}
\usepackage{amssymb}
\usepackage{epstopdf}
\usepackage{setspace}
\usepackage{leftidx}

\usepackage{xcolor}
\usepackage[colorlinks,linkcolor=black, citecolor=black]{hyperref}
 \usepackage{url}

\theoremstyle{plain}
\newtheorem{theo}{Theorem}[section]
\newtheorem{lem}[theo]{Lemma}
\newtheorem{cor}[theo]{Corollary}
\newtheorem{prop}[theo]{Proposition}

\numberwithin{equation}{section}

\theoremstyle{definition}

\newtheorem{remark}[theo]{Remark}
\newcommand{\bequ}{\begin{equation}}
\newcommand{\eequ}{\end{equation}}
\newcommand{\bali}{\begin{align}}
\newcommand{\eali}{\end{align}}

\def\e{\varepsilon}

\def\vf{\varphi}
\def\a{\alpha}

\def\d{\delta}

\def\D{\Delta}

\def\k{\kappa}

\def\LL{\mathcal L}
\def\pp{\partial}
\def\l{\rho}
\def\L{\Lambda}

\def\s{\sigma}

\def\x{\times}
\def \R{\mathbb R}
\def \N{{\mathbb N}}
\def\E{\mathbb E}

\def \T{\mathcal T}

\def \P{\mathbb P}
\def\ov{\overline}
\def\un{\underline}

\def\om{\omega}
\def\Om{\Omega}
\def\OO{\mathcal O}
\def\A{\mathcal A}

\def\W{\mathcal W}
\def\wh{\widehat}
\def\wt{\widetilde}
\def\({\biggl(}
\def\){\biggr)}

\def\<{\bold\langle}
\def\>{\bold\rangle}

\def\LL{{{\mathcal L}}}
\def\RR{{{\mathcal R}}}
\def\QQ{{{\mathcal Q}}}
\def\SS{{{\mathcal S}}}
\def\PP{{{\mathcal P}}}
\def\XX{\mathcal{X}}

\def\M{\widetilde {M}}

\def\S{\mathcal{S}}

\usepackage{mathtools,mathabx}
\newcommand\AR[1]{\makebox[0pt][l]{$#1$}\kern0.5em\raisebox{1.5ex}{$\curvearrowleft$}} 

\usepackage{graphicx}
\usepackage{float}
\usepackage{stmaryrd}

\begin{document}
\title[A family of stable diffusions]{A family of stable diffusions}
\author{Fran\c cois Ledrappier  and  Lin Shu}
\address{Fran\c cois Ledrappier,  Sorbonne Universit\'e, UMR 8001, LPSM, Bo\^{i}te Courrier 158, 4, Place Jussieu, 75252 PARIS cedex
05, France}\email{fledrapp@nd.edu}
\address{Lin Shu,  LMAM, School
of Mathematical Sciences, Peking University, Beijing 100871,
People's Republic of China} \email{lshu@math.pku.edu.cn}

\thanks{The second author was  partially  supported by  NSFC (No.11331007 and No.11422104).}

\maketitle
\begin{abstract}
Consider a  $C^{\infty}$ closed connected Riemannian manifold $(M, g)$ with negative curvature. The unit tangent bundle $SM$ is foliated by  the (weak) stable foliation  $\mathcal{W}^s$ of the geodesic flow. 
 Let $\Delta^s$ be the leafwise Laplacian for $\mathcal{W}^s$ and let $\ov{X}$ be the geodesic spray,  i.e., the vector field that generates the geodesic flow.  For each $\l$, the operator $\mathcal{L}_{\l}:=\Delta^s+\rho \overline{X}$ generates a  diffusion for  $\mathcal{W}^s$.  We show that, as $\l\to -\infty$, the unique stationary probability measure for the leafwise diffusion  of $\mathcal{L}_{\l}$ converge  to the normalized Liouville measure on $SM$. \end{abstract}

\section{Statement of the result}
Let  $(M, g)$ be an $m$-dimensional closed connected negatively curved $C^{\infty}$ Riemannian  manifold. 
We shall study a class of probability measures on the unit tangent bundle $SM$  which interpolates between the  Burger-Roblin  measure (whose transversal distribution in the weak unstable leaves  is the same as the one for the maximal entropy measure  of the geodesic flow)  and the normalized Liouville measure.

 Let $\wt{g}$ be the the $G$-invariant extension of $g$ to the universal cover space $\M$.  The fundamental group $G=\pi_1(M)$ acts on  $(\wt{M}, \wt{g})$ as isometries such that $M=\wt{M}/G$. Let $\pp\M$ be  the \emph{geometric boundary} of $\M$,   i.e.,   the collection of equivalent classes of unit speed  geodesic rays that remain a bounded distance apart. Since $\wt{g}$ is negatively curved, there is a natural homeomorphism from  $\partial \wt{M}$ to   the unit sphere  $S_{x}\wt{M}$ in the tangent space at $x\in\wt{M}$,  sending $\xi$ to the initial vector of the  geodesic  ray starting from $x$ in the equivalent class  of $\xi$ (\cite{EO}).  Hence we  identify the unit tangent bundle $S\wt{M}=\bigcup_{x\in \M}S_x\wt{M}$  with $\M\times \pp\M$. 

 For each
${\bf v}=(x, \xi)\in S\M$, its \emph{(weak) stable manifold} for the geodesic flow $\{{\bf
\Phi}_t\}_{t\in \Bbb R}$ on $S\wt{M}$, denoted $\wt W^s({\bf v})$,  is the collection of initial
vectors  of  geodesic rays in the equivalent class of $\xi$  and can be
identified with $\M\times\{\xi\}$.  The collection of  $\wt W^s({\bf v})$ form the stable foliation $\wt\W^s$ of $S\M$.  Extend the action of $G$ continuously to $\pp\M$.  Then $SM$   can be identified with the quotient of
$\M\times
\partial\M$ under the diagonal action of $G$.  
  Since  $\psi(\wt W^s({\bf v})) = \wt W^s({D\psi ({\bf v}}))$ for $\psi \in G$,  the collection of quotients of 
 $\wt W^s({\bf v})$ defines a lamination $\W^s$ on $SM$, the so-called \emph{(weak) stable foliation} of $SM$.
 The leaves of  $\W^s$ are discrete quotients of $\M$, which are naturally endowed with the Riemannian metric induced from $\wt g$.  For ${v}\in SM$, let $W^s(v)$ be the leaf of $\W$ containing $v$.
Then $W^s(v)$ is a $C^{\infty}$
immersed submanifold of $SM$ depending H\"{o}lder continuously on $v$ in the
$C^{\infty}$-topology (\cite{SFL}). 

Let  $\mathcal{L}$ be a Markovian operator  (i.e.,  $\mathcal L 1 = 0 $)  on (the smooth functions on) $SM$
with continuous coefficients.  It is  said to be \emph{subordinated to the stable foliation $\mathcal{W}^s$}, if for every smooth function $f$ on $SM$, the value of $\mathcal{L}(f)$  at $v\in SM$ only depends on the restriction of $f$ to $W^s(v)$.   A Borel probability measure ${\bf m}$ on $SM$ is called
\emph{$\LL$-harmonic}  if it satisfies
\[
\int \LL(f)\ d{\bf m}=0
\]
for every smooth function $f$ on $SM$. 
Extend  $\LL$  to be a $G$-equivariant operator on $S\wt{M}=\M\times
\partial \M$, which we shall denote with the same symbol, and,  for ${\bf{v}}=(x, \xi)\in S\M$, let
$\LL^{\bf v}$ denote the laminated operator of $\LL$ on $\wt W^s({\bf
v})=\M\times \{\xi\}$. Call   $\LL$   \emph{weakly coercive},  if its lifted leafwise operators $\LL^{{\bf v}}$, ${\bf v}\in S\M$,
are weakly coercive in the sense that there are a number $\e>0$
(independent of ${\bf v}$)  and, for each $\bf v$, a positive
$(\LL^{{\bf v}}+\e)$-superharmonic function $F$ on $\M$ (i.e.,  $(\LL^{{\bf v}}+\e)F\leq 0$).   It is known that for  a weakly coercive operator, there exists a unique harmonic  measure (\cite{Ga}, \cite{H2}).

One classical example of  weakly coercive operator is $\mathcal{L}=\Delta^s$,  the laminated Laplacian for $\mathcal{W}^s$, whose unique $\mathcal{L}$-harmonic  measure is always referred to as  the \emph{harmonic measure}  (\cite{Ga}).  Many interesting open problems in dynamics are concerned  with  the relationship of  the harmonic measure  with  the normalized Liouville measure  and the normalized  maximal entropy measure for the geodesic flow (Bowen-Margulis measure), and the applications of these relations  to the characterizations  of  the locally symmetric property of the underlying space (see \cite{K, Su} and see  also  \cite{Kai, L5, Y} for more descriptions).

  In this paper, we are interested in the family  $\LL_\rho=\Delta^s+\rho \ov X$,
 where  $\rho$ is a real number and 
$\ov X$ is the geodesic spray. 
 Since $\ov X$ is tangent to the stable manifold,
the operators $\LL_\rho $ are subordinated to the stable foliation.   

Let $V$  denote the \emph{volume entropy} of  $(M,g)$:
\[ V \; = \; \lim\limits_{r \to +\infty } \frac {\log {\rm{Vol}} (B(x, r))}{r} ,\]
where  $B(x, r) $ is the ball of radius $r$ in $(\M, \wt{g})$ and ${\rm Vol}$ is the volume.  The volume entropy coincides with the topological entropy of the geodesic flow on $SM$ since  $g$ has negative sectional curvature (\cite{Man}). 
For $\rho < V$, the operator $\LL _\rho $ is weakly coercive (\cite{H2}) and hence  there is a unique $\LL_{\rho} $-harmonic  measure, which we  will denote by ${\bf m}_{\rho}$.

Clearly,  ${\bf m}_{0}$ is the classical harmonic measure.  When $\rho \to {V}$,  ${\bf m}_{\rho}$ tends to the  Burger-Roblin  measure ${\bf m}_{BR},$  the unique harmonic  measure for  the Laplacian subordinated to  the  strong 
stable foliation  (\cite[Proposition 4.10]{LS}, the uniqueness of such a measure is due to  Kaimanovich (\cite{Kai2})). 
 When $\rho \to -\infty$,   the main result of this paper is:
\begin{theo}\label{mainth} Let  $(M, g)$ be an $m$-dimensional closed connected negatively curved $C^{\infty}$ Riemannian  manifold. As $\rho \to - \infty$, the $ \LL_{\rho}$-harmonic  measure ${\bf m}_\rho $ converge  to the normalized  Liouville  measure on $SM$.\end{theo}

Roughly speaking,  since  the  measure ${\bf m}_\rho $ is $\LL_\rho$-harmonic, it is also stationary for the operator $-\overline X - (1/\rho)  \D^s$ (see Section 2 for a precise definition). In particular, any limit measure 
of the family ${\bf m}_\rho$ as $-1 /\rho \to 0$ is invariant under the (reversed) geodesic flow. For  a limit of random perturbations of a conservative  Anosov flow,  the convergence of 
the stationary measures  to  a  SRB measure
has been shown by several authors, in particular Kifer (\cite{Ki}), under the condition that the operator $\D$ is hypoelliptic, so that the Markov kernels have a density with respect to Lebesgue on $SM.$  We cannot apply this  to show Theorem \ref{mainth} since in our case, the operators are subordinated to the stable foliation and the Markov kernels ${\bf p}_\rho (t, (x, \xi), d(y, \eta))$ are singular. Another approach by Cowieson-Young (\cite{CY}) uses the variational principle from thermodynamical formalism and we show that such an approach can be used in our case in spite of the singularity of the Markov kernels.  We shall show any  limiting measure ${\bf m}$ of ${\bf m}_{\rho}$ (as $\rho\to -\infty$) satisfies  Pesin  entropy formula for the geodesic flow. Theorem \ref{mainth} follows since the normalized Liouville measure  on $SM$ is indeed characterized by Pesin formula among invariant measures for the geodesic flow (\cite{BR}).
More precisely,  we will  define a stochastic flow on a bigger space and consider a special stationary measure $\ov{\bf m}_\rho$ for that stochastic flow that projects to ${\bf m}_\rho $ on $SM$.  We then  introduce a relative entropy like quantity  $h_{\rho}^s$ for $\ov{\bf m}_\rho$ and  show $h_{\bf m}$, the entropy of ${\bf m}$ for the reversed geodesic flow, satisfies 
\begin{equation}\label{upper-semi}
h_{\bf m}\geq \limsup_{\rho\to -\infty}h_{\rho}^s.
\end{equation}
This  can be done (see Proposition \ref{mainprop}) along the lines of  Cowieson-Young (\cite{CY})  and Kifer-Yomdin (\cite{KY}) for the upper semi-continuity of the relative entropy. To conclude Theorem \ref{mainth}, we  verify that  $\limsup_{\rho\to -\infty}h_{\rho}^s$ has a lower bound  given by Pesin  entropy integral for ${\bf m}$ using the SRB like  properties of $\ov{\bf m}_\rho$ (see Proposition \ref{Thieullen}) and their  nice convergence property  inherited from  our stochastic flow system (see Proposition \ref{con-geo-flow}). 

We arrange the paper as follows. In Section 2, we will give preliminaries on the  properties of the $\mathcal{L}$-harmonic measures and the dynamics of the associated  stochastic flows.  In Section 3, we will introduce the random system to define  $h_{\rho}^s$ and reveal its relation with Pesin entropy formula. The upper semi-continuity equality (\ref{upper-semi}) will be shown in the final section. 

\section{Harmonic measure and stochastic flow}

We begin with some basic  understanding of  the $\LL_{\rho}$-harmonic measure ${\bf m}_{\rho}$ ($\rho<V$) by analyzing  the dynamics of its $G$-invariant extension on $\M\times \partial \M$, which is denoted by  $\wt{\bf m}_{\rho}$.

Consider the $G$-equivariant extension of  $\LL_\rho$  to  $S\wt{M}=\M\times
\partial \M$, which we shall  denote by the same symbol. It defines
a Markovian family of probabilities on $\wt
{\Omega}_{+}$, the space of paths of  $\wt{\om}: \Bbb R_+\to S\M$  (where $\Bbb R_+:=[0, +\infty)$),  equipped with the smallest
$\sigma$-algebra $\A$ for which the projections $R_t:\ \wt{\om}\mapsto
\wt{\om}(t)$ are measurable. Indeed, for ${\bf{v}}=(x, \xi)\in S\M$,  the laminated operator 
$\LL_{\rho}^{\bf v}$ on  $\wt W^s({\bf
v})$ can be regarded as an operator on $\M$ with corresponding
heat kernel functions  $p_{ \rho }^{\bf v}(t, y, z)$, $t>0, y, z\in \M$.
Define
\[
{\bf p_\rho }(t, (x, \xi), d(y, \eta))=p_{ \rho }^{\bf v}(t, x,
y)\, d{\rm{Vol}}(y)\,\delta_{\xi}(\eta),
\]
where $\delta_{\xi}(\cdot)$ is the Dirac function at $\xi$. Then the
diffusion process on $\wt W^s({\bf v})$ with infinitesimal operator
$\LL_{\rho}^{\bf v}$ is given by a Markovian family $\{\P_{\rho }^{\bf
w}\}_{{\bf w}\in \M\times \{\xi\}}$, where for every $t>0$ and
every Borel set $A\subset \M\times \partial\M$ we have
\[
\P_{\rho }^{\bf w}\left(\{\wt{\om}:\ \wt{\om}(t)\in A\}\right)=\int_{A}{\bf p}_\rho (t, {\bf
w}, d(y, \eta)).
\]

\begin{prop}\label{harmonic-mea}(\cite{Ga,H2}) With the above notations, the following  are true  for $\rho<V$.
\begin{itemize}
\item[i)] The measure $\wt{{\bf m}}_\rho$ satisfies, for all $f\in C^{2}(\wt{M}\times \partial
  \M)$ with compact support, 
  \[
\int_{\scriptscriptstyle{\M\times
\partial\M}}\left(\int_{\scriptscriptstyle{\M\times \partial\M}}f(y,
\eta){\bf p}_\rho(t, (x, \xi), d(y, \eta))\right) d\wt{{\bf m}}_\rho(x,
\xi)=\int_{\scriptscriptstyle{\M\times \partial\M}}f(x, \xi)\
d\wt{{\bf m}}_\rho(x, \xi).
  \]
  \item[ii)] The measure $\wt \P_\rho=\int \P_{\rho}^{\bf v}\ d\wt{{\bf m}}_\rho({\bf v})$ on
  $\wt{\Om}_{+}$ is invariant under all the shift maps   $\{\sigma_t\}_{t\in \Bbb R_+}$ on $\wt{\Om}_+$,   where $\sigma_t(\wt{\om}(s))=\wt{\om}(s+t)$  for $s\in \Bbb R_+$ and $\wt{\om}\in
  \wt{\Om}_{+}$.
   \item[iii)] The measure $\wt{{\bf m}}_\rho$ can be
expressed locally at ${\bf v}=(x, \xi)\in S\M$ as $d\wt{{\bf
m}}_{\rho}=k_\rho(y, \eta)(dy\times d\nu_{\rho}(\eta))$, where $\nu_{\rho}$ is a finite
measure on $\partial \M$ without atoms and, for $\nu_{\rho}$-almost every $\eta$, $k_\rho(y,
\eta)$ is a positive function on $\M$ satisfying the equation 
\begin{equation}\label{m-lam-equ}\Delta(k_\rho(y,
\eta))-\rho{\rm{Div}}(k_\rho(y, \eta)\ov X(y, \eta))=0,
\end{equation}
where we continue to use $\ov{X}$ to denote the geodesic spray for  $S\M$. 
\end{itemize}
\end{prop} 

\begin{remark}\label{rem-1}
Let ${\bf m}$ be any weak* limit of the probability measures ${\bf m}_\rho $ on  $SM$ as $\rho \to -\infty$ and let $\wt{\bf m}$ be the $G$-invariant extension of ${\bf m}$ to $\M\times \partial \M$. Clearly, Theorem \ref{mainth} follows if we can  show $\wt{\bf m}$ has absolutely continuous conditional measures on leafs $\M\times \{\xi\}$. But this does not follow directly  from equation (\ref{m-lam-equ}) since  the Harnack inequality used for each $\rho$ finite is worse and worse when $\rho$ goes to $-\infty$ and hence we have less and less  control of the  density functions $k_\rho$. 
\end{remark}

  For Theorem \ref{mainth}, we will further explore the  invariant dynamics of ${\bf m}_{\rho}$ from the stochastic flow point of view and use it to establish  the entropy formula  for the limit measures. 
  
  We first recall some classical results  from the theory of  Stochastic Differential Equations (SDE).
Let $\{B_t=(B^1_t, \cdots, B^d_t)\}_{t \in \Bbb R_+} $ be a $d$-dimensional Euclidean Brownian motion starting from the origin with the Euclidean Laplacian  generator  (so the covariance matrix is  $2t {\rm Id}$) and let $(\Om, \P)$ denote the corresponding Wiener space. Let ${\bf X}=(X_0, X_1, \cdots, X_d)$, where $\{X_i\}_{i\leq d+1}$ are  bounded vector fields on a smooth finite dimensional Riemannian
manifold $({\rm N}, \langle\cdot, \cdot\rangle)$.   The pair $({\bf X}, \{B_t\}_{t\in \Bbb R_+})$  consists of   a  \emph{stochastic dynamical system  (SDS)} on ${\rm N}$  and it is   $C^{j}$ ($j\geq 1$ or $j=\infty$) if all $X_i$ are $C^j$ bounded  (\cite{El}).  An ${\rm N}$-valued semimartingale   $\{{x}_t\}_{t\in \Bbb R_+}$ defined up to a stopping time ${\bf e}^{x_0}$ is said
to be a solution of the following Stratonovich SDE 
\begin{equation}\label{SDE-x}
d{x}_t(\om)=X_0({x}_t(\om))\, dt + \sum_{i=1}^{d}X_i({x}_t(\om))\circ dB_{t}^{i}(\om), 
\end{equation}
 if for all $f\in C^{\infty}({\rm N})$, 
\[
f({x}_t(\om))=f({x}_0(\om))+\int_{0}^{t}  X_0 f({x}_s(\om))\, ds +\int_{0}^{t} \sum_{i=1}^{d} X_i f({x}_s(\om))\circ dB_{s}^{i}(\om), \ \forall  0\leq t< {\bf e}^{x_0}(w).
\] 
The  solution to (\ref{SDE-x}) always exists  and is essentially unique when all $X_i$'s   are $C^1$ bounded (\cite{El}). Moreover, for   $\P$ almost all ${\om}$,  the  mapping 
\[
F_t(\cdot, {\om}):\ x_0({\om})\mapsto x_t({\om})
\]
has the following property.

\begin{prop}\label{El-Ku}(\cite[Chapter VIII]{El}) Let $({\bf X}, \{B_t\}_{t\in \Bbb R_+})$ be a $C^j$ SDS  on ${\rm N}$, where $j\geq 1$ or $j=\infty$. There is a version of the explosion time map $x\mapsto {\bf e}^x$, defined for $x\in {\rm N}$, and a version of $\{F_t(x, {\om})\}$, defined when $t\in [0, {\bf e}^x({\om}))$, such that if ${\rm N}(t, {\om})=\{x\in {\rm N}:\ t<{\bf e}^x(w)\}$, then the following are true  for each $(t, {\om})\in \Bbb R_+\times \Omega$. 
\begin{itemize}
\item[i)] The set ${\rm N}(t, {\om})$ is open in ${\rm N}$. 
\item[ii)] For almost all $w$, $x_0\in {\rm N}$ and  $0\leq t<t'<{\bf e}^{x_0}(w)$,  we have the cocycle equality 
\[
F_{t'}(x_0, \om)=F_{t'-t}(x_{t}, \sigma_t(\om))\circ F_{t}(x_0, \om),
\]
where $\s_t $ is the shift transformation on $\Om$: 
\[\s _t \left((B^1_s, \cdots, B^m_s)_{s \geq 0}\right)  = \left((B^1_{t+s}, \cdots, B^m_{t+s})_{s \geq 0}\right)   - (B^1_t, \cdots, B^m_t).\]
\item[iii)] The map $F_t(x, {\om}):\ {\rm N}(t, {\om})\to {\rm N}$ is $C^{j-1}$ (or $C^{\infty}$ when $j=\infty$) and is  a diffeomorphism onto an open subset of ${\rm N}$. Moreover, the map $\tau\mapsto F_{\tau}(\cdot, {\om})$ of $[0, t]$ into $C^{j-1}$ (or $C^{\infty}$ when $j=\infty$) mappings of ${\rm N}(t, {\om})$  is continuous.
\item[iv)] For $1\leq l\leq j-1$,  denote by $D^{(l)}F_t(\cdot, {\om})$ the  $l$-th tangent map of $F_t$.  Then, for  any $q\in [1, \infty)$, there is a bounded function $c_l(t, q)$, which depends on $t, m, q$, and the bounds of $\{{\bf \nabla}^{\iota}X_0\}_{\iota\leq l}$, $\{{\bf \nabla}^{\iota}X_i\}_{1\leq i\leq d, \iota\leq l+1}$ and $\{{\bf \nabla}^{\iota-1}R\}_{1\leq i\leq d, \iota\leq l+1}$ such that 
$\|[D^{(l)}F_t(\cdot, {\om})]\|_{L^q}<c_l(t, q)$, where $\|\cdot\|_{L^q}$ is the $L^q$-norm and ${\bf \nabla}^{\iota}$ denotes the $\iota$-th covariant derivative and $R$ is the curvature tensor. \end{itemize}
\end{prop}

When  ${\rm N}(t, {\om})\equiv {\rm N}$, the solution process $\{x_t\}$  to (\ref{SDE-x}) is said to be {\it{non-explosive}}. In this case,   the maps $\{F_t(\cdot, {\om})\}_{t\in \Bbb R_+}$ induce a kind of semi-flow   on ${\rm N}$, which we shall call the \emph{stochastic flow} associated to the SDS  $({\bf X}, \{B_t\}_{t\in \Bbb R_+})$   or (\ref{SDE-x}). A direct consequence of  Proposition \ref{El-Ku} is the following  regularity of  a one-parameter  family  of stochastic flows.

\begin{cor}\label{cor2}Let  $({{\bf X}}^{a},\{B_t\}_{t\in \Bbb R_+})$  be a one-parameter family  of  SDS on ${\rm N}$ with ${\rm N}^a(t, \om)\equiv {\rm N}$. 
Assume ${{ X}}^{a}_i$'s are all $C^k$ $(k\geq 1\, {\mbox or}=\infty)$ on ${\rm N}\times {\rm A}$ in the product differentiable structure. Then for any $t>0$,  and $j\leq k-1$, $a\mapsto F_{t}^{a}(\cdot, {\om})$ is $C^j$  in the space of  $C^{k-1-j}$ maps of ${\rm N}$. 
 \end{cor}
\begin{proof}Let $x^{a}_t$ be the solution for the SDS  $({{\bf X}}^{a},\{B_t\}_{t\in \Bbb R_+})$.  Then  $(x^{a}_t, a)$ solves  the new SDS  $(({{\bf X}}^{a}, 0), \{B_t\}_{t\in \Bbb R_+})$ on ${\rm N}\times {\rm A}$. The regularity in $a$ is a straightforward application of Proposition \ref{El-Ku} by treating $a$ as a part of the initial value. 
\end{proof}

 Corollary \ref{cor2} does not apply when  we only have H\"{o}lder continuity of  ${\bf X}^{a}$ in $a$. However,  it is still possible to  discuss the regularities of $a\mapsto F_{t}^{a}(\cdot, {\om})$  by using  one criterion from Kolmogorov:

\begin{prop}(cf. \cite[Theorem 1.4.1]{Ku2})\label{Kol-criterion} 
 Let  $T>0$ and let $\{\mathcal{Y}^{a}_{t}({\om})\}_{t\in [0, T], a\in \mathtt A}$ be a one parameter family  of random processes on a complete metric space, where $\mathtt A$ is some bounded $n$-dimensional Euclidean domain. Suppose there are positive constants $\flat, \flat_0, \flat_1, \cdots, \flat_n$,  with $\sum_{i=0}^n(\flat_i)^{-1}<1$,  and ${\mathtt C}_0(\flat)$ such that for  all $t, t'\in [0, T]$ and $a=(a_1, \cdots, a_n), a'=(a'_1,\cdots, a'_n)\in \mathtt  A$, 
\[
\E\left[\left|\mathcal{Y}^{a}_t-\mathcal{Y}^{a'}_{t'}\right|^{\flat}\right]\leq  {\mathtt C}_0(\flat)\left(|t-t'|^{\flat_0}+\sum_{i=1}^n |a_i-a'_i|^{\flat_i}\right),\]
then  $\mathcal{Y}^{a}_t$ has a continuous modification with respect to the parameter $(t, a)$. 
\par
Let $\beta_i, i=0, \cdots, n$,  be arbitrary positive numbers less than $\flat_i(1-\sum_{0}^{n}(\flat_i)^{-1})/\flat$. Then for any hypercube $\mathtt D$ in $\mathtt  A$, there exists a positive random variable ${\rm k}(\om)$ with $\E[{\rm k} (\om)^{\flat}]<\infty$ such that  for any $t, t'\in [0, T]$ and $a, a'\in \mathtt D$, 
\[
\left|\mathcal{Y}^{a}_t-\mathcal{Y}^{a'}_{t'}\right|\leq  {\rm k} (\om)\left(|t-t'|^{\beta_0}+\sum_{i=1}^n |a_i-a'_i|^{\beta_i}\right).
\]
\end{prop}

Next, we consider  $\rho <0, \, \e := 1/\sqrt{-\rho} $ and  $\LL'_\e := -\ov X + \e^2 \D^s$. 
Extend  $\LL'_\e$ to be  a $G$-equivariant operator on   $S\wt{M}$, which we shall  denote by  the same symbol.  Its  associated leafwise diffusions  can be visualized using  the classical  Eells-Elworthy-Malliavin construction.

Recall that, for ${\bf v} = (x, \xi)  \in S\M,$  we have identified the stable manifold $\wt W^s({\bf v}) $ with $\M \x \{\xi \} $ and endowed it with the Riemannian metric on $\M$. In the  same way, we can identify an orthogonal frame in the tangent space $T_{{\bf v}}\wt W^s$ with $O_x \x \{\xi \} $, where $O_x  = (e_1, \cdots, e_m) $ is  an element in $\mathcal{O}_x(\M)$, the collection  of the orthogonal frames in $T_x\M.$ Set $\OO ^s(S\M)$ for the bundle of such stable orthogonal frames:
\[ \OO^s (S\M):=  \left\{ (x, \xi ) \mapsto O_x \x \{\xi \}:\  O_x= (e_1, \cdots, e_m) \in \mathcal{O}_x(\M),\, x\in \M \right\}.\]
We carry to $\OO ^s(S\M)$ all the Riemannian geometry from $\OO (\M)=\bigcup_{x\in \M}\OO_x(\M)$.  In particular, if $H_x$ denotes the horizontal lift from $T_x\M $ to $T_{O_x}\OO(\M) ,$ we can define the horizontal lift $\wh H_{{\bf v}}$ from $T_{{\bf v}}W^s$ to $T_{O_{x}, \xi}\OO ^s(S\M) $ by $\wh H_{{\bf v}}(w,\xi) = (H_x(w), \xi)$ for $w\in T_x\M$. 

Let $\{(B^1_t, \cdots, B^m_t)\}_{t\in \Bbb R_+} $  be an $m$-dimensional Euclidean Brownian motion starting from the origin with the Euclidean Laplacian  generator (and covariance matrix  $2t {\rm Id}$) and let $(\Om, \P)$ be the Wiener space.  Set  $\wh {X}$ as  the horizontal lift of $\ov X$ to $T\OO ^s(S\M)$.  We can  realize the diffusion for $\LL'_\e $ as the projection to $S\M$ of the non-explosive solution process $\{{\rm u}_t\}_{t\in \Bbb R_+}$  (the non-explosiveness follows since $(\M, \wt{g})$ has Ricci curvature uniformly bounded from below)  to   the Stratonovich SDE  on $\OO ^s(S\M)$:
\begin{equation}\label{basicSDE} du_t = -\wh {X}(u_t)\: dt  + {\e} \sum _{i = 1} ^m \wh H(u_t(e_i)) \circ dB^i_t. \end{equation}

Let $\wh \pi : \OO ^s(S\M) \to S\M $ be the natural projection and denote $\wh{\W}^s$  for the  foliation of $\OO ^s (S\M)$ that projects on $\wt\W ^s$. Let  $D^\infty (\OO^s S\M) $ be the space of homeomorphisms  of $\OO^s ( S\M)$ that preserve the leaves of $\wh{\W}^s$ and are $C^\infty $-diffeomorphisms along the leaves. We endow $D^\infty (\OO^s S\M) $ with the $C^{0,\infty} $ topology: $\vf,\vf' \in D^\infty (\OO^s S\M) $ are close if, for all $r >0$, the $r$-germs  of $\vf$ and $\vf'$ are uniformly close on compact sets  and the $r$-germs  of $\vf^{-1}$ and $(\vf')^{-1}$ are uniformly close on compact sets.

\begin{prop}\label{stochasticflow}
With the above notations,  for  $\P$-a.e.   $\om \in \Om$,  for all $\e >0$, $t\geq 0$, there exists  $ \vf_{\e,t} (\om) \in  D^\infty (\OO ^s S\M)  $ such that the following hold true. 
\begin{enumerate} 
\item[i)] For all $u \in \OO ^s (S\M),$ $(\om, t )\mapsto \vf_{\e,t} (\om) (u) $ is a solution of the equation (\ref{basicSDE}); in particular, for all $T\geq 0$, $\om\mapsto \vf_{\e,T} (\om) $ is measurable with respect to the $\s$-algebra generated by $(B^1_t, \cdots, B^m_t), {0 \leq t \leq T} .$
\item[ii)] For   almost all $\om ,$ 
all $t,s\geq 0$, $\vf_{\e, t+s} (\om)  = \vf_{\e, t} (\s _s (\om)) \circ \vf_{\e, s} (\om)$.
\item[iii)] For all $\psi\in G$, $D\psi\circ  \vf_{\e,t}(\omega)= \vf_{\e,t}(\omega)\circ D\psi.$
 \item[iv)] The map  $\e \mapsto  \vf_{\e,t} (\om) $ is continuous in $D^\infty (\OO ^s S\M).$ 
\item[v)] For  fixed $r\in \Bbb N, t\geq 0$, 
 \begin{equation}\label{B-r-bound}
 \E\left[\max _u \big\| \vf _{\e,t}(\om)(u)|_{\wh W^s (u)} \big\|_{C^r}\right]<+\infty. 
 \end{equation}
\end{enumerate}\end{prop}
\begin{proof}Since both $\wh { X}$ and $\wh H$ are  tangent to $\wh{\W}^s$,   the solution to (\ref{basicSDE}) is constrained in  $\wh{\W}^s$.  For fixed $\xi$ and $\e$, equation (\ref{basicSDE}) can be seen as a SDE on $\OO (\M) \x \{\xi \}$ and is solvable with infinite explosion time. Hence  properties i) and ii) are given by Proposition \ref{El-Ku}. Property iii) follows from the uniqueness of the solution to (\ref{basicSDE}). 
 Considering $\e$ as a parameter, we get the continuity of the solution in $\e$ by Corollary \ref{cor2}.  Considering $\xi$ as a parameter, the leaves of $\wh{\W}^s$ and  $\wh {X}$, $\wh H$ vary H\"{o}lder continuously with respect to $\xi$. Hence,  by a standard estimation using Burkholder inequality and Gronwall lemma and applying  Proposition \ref{Kol-criterion}, we can  obtain the continuity of the solution to (\ref{basicSDE}) in $\xi$, so that we can consider it as an element of $D^\infty (\OO ^s S\M)$.  This shows iv).  Finally we show v).   Using a fundamental domain for the action of $G$ on $\M$, we may regard  $\OO^s(SM)$ as a subset of $\OO^s(S\M)$. By the $G$-equivariance property of the diffusion, we can restrict $u$ in the left hand side of  (\ref{B-r-bound}) to  $\OO^s(SM)$.  By continuity of $\vf _{\e, t}(\om)(u)|_{\wh W^s (u)}$ in $u$, the compactness of  $\OO^s(SM)$ and  Proposition \ref{Kol-criterion}, for (\ref{B-r-bound}), it suffices to show for each 
   $u\in\OO^s(SM), r\in \Bbb N$ and $ t>0$, 
\begin{equation}\label{B-r-continuity}
\E\left[ \big\| \vf _{\e,t}(\om)(u)|_{\wh W^s (u)} \big\|_{C^r}\right]<+\infty. 
\end{equation}
This is an application of Proposition \ref{basicSDE} iv) by using the SDE (\ref{basicSDE}). 
 \end{proof}

Equation (\ref{basicSDE}) for $\e = 0 $ is the ordinary differential equation $du_t = -\wh {X}(u_t)\; dt $. Its solution is the extension $\{\wh {\bf \Phi }_{-t}\}_{t\in \Bbb R}$  of the reversed geodesic flow to $\OO ^s (S\M)$ by parallel transportation  along the geodesics,  and is  called the \emph{reversed stable frame flow}.  Write $\nu _\rho$ for the probability measure on $D^\infty (\OO^s S\M) $ that is the distribution of $\vf _{\e, 1} =\vf _{1/\sqrt{-\rho},1}$ in Proposition \ref{stochasticflow}.  Every element $\vf \in D^{\infty } (\OO^s S\M)$ preserves each leaf $\wh W^s (u) $  and is a $C^\infty $ diffeomorphism along it. We write $J(\vf ,u)$ for the Jacobian determinant of  the tangent map  of $\vf \large|_{\wh W^s (u)} $ at $u$.  For later use, we state a proposition  concerning the limit behavior of $\vf _{1/\sqrt{-\rho},t}$ when $\rho\to -\infty$. 

\begin{prop}\label{con-geo-flow} With the above notations, the following are true. 
\begin{itemize}
\item[i)] For  $\P$-a.e.  $\om \in \Om$, all $t>0$, as $\rho \to -\infty,$  $\vf _{1/\sqrt{-\rho},t} (\om) $ converge to $\wh {\bf \Phi }_{-t}$ in $D^\infty (\OO ^s S\M)$,  in particular,  $\vf _{1/\sqrt{-\rho},1}$ converge to the time 1 map of the reversed  stable frame flow.
\item[ii)] For any $-\infty<\rho<0$ and $ r, \mathsf M $ positive integers, 
\begin{equation}\label{r-derivatives}B_{r,\mathsf M}(\rho) := \E\left[\max _u \big\| \vf _{1/\sqrt{-\rho},\mathsf M}(\om)(u)|_{\wh W^s (u)} \big\|_{C^r}\right]  < +\infty {\textrm { and } } B_{r, \mathsf M} :=\limsup_{\rho \to  -\infty } B_{r, \mathsf M}(\rho) \; < \; + \infty. \end{equation}
\item[iii)] For any  $r\in \Bbb N$, 
\[
\lim\limits_{\rho\to -\infty}\int  \max _u \left|\| \vf|_{\wh W^s (u)} \|_{C^r} -\| \wh {\bf \Phi }_{-1}|_{\wh W^s (u)} \|_{C^r} \right| \, d\nu _\rho (\vf )=0.
\]
\item[iv)]We have 
\[\lim\limits_{\rho\to -\infty}\int  \log J(\vf , u) \, d\nu_\rho (\vf) =\log J(\wh{\bf\Phi}_{-1}, u)\]
and the convergence is locally uniform in $u$. 
\end{itemize}
\end{prop}
\begin{proof}The proof of continuity of the solution to  (\ref{basicSDE})  in  Proposition \ref{stochasticflow}  extends to $\e = 0$. This shows i).  When $\rho\to -\infty$,  $\e=1/\sqrt{-\rho}\to 0$.  For ii),  note that 
$B_{r, \mathsf M} (\rho)$ is finite by Proposition \ref{stochasticflow}  v) (applied for $t = \mathsf M$).  Then  $B_{r,\mathsf M}$ is also finite by using the continuity in $(\e, u)$ in the estimation of the expectation in (\ref{B-r-continuity}) in  the proof of Proposition \ref{stochasticflow}  v).  Similarly, we have the continuity  in $\e$ of the tangent maps  (and their derivatives in $\e$) of the solution to  (\ref{basicSDE}).  Following Proposition \ref{El-Ku} iv), it is easy to deduce from (\ref{basicSDE}) the continuity in $\e$ of the norm of the tangent maps and of the Jacobian of the first order tangent map. This shows iii) and iv). 
\end{proof}

It follows from Proposition \ref{stochasticflow} i) and ii) that we can consider $\vf _{\e,n}, n\in \N$,   as an independent product of the  homeomorphisms $\vf _{\e, 1} $ and that we can apply the theory of independent random mappings.    Let $\wh{\pi}$ be the projection map from $\OO^s (S\M)$ to $S\M$.  For  any $C^2$ compactly supported function $f$ on $S\M$,  $(x,\xi) \in S\M$ and  any frame $u \in \OO^s (S\M)$  in the fiber  $\wh{\pi} ^{-1} (x,\xi)$, we have 
\begin{equation}\label{stationarity} \int _{S\M} f (y, \eta) \, d{\bf p}_{\rho}(\e, (x, \xi), d(y, \eta)) \; = \; \int _{D^\infty (\OO^s S\M)} f ( \ov{\pi} \vf (u) )\, d\nu _\rho (\vf ) . \end{equation}
Let   $\wh{\bf m}_\rho$ be the measure on  $\OO^s (S\M) $  that projects on 
$\wt{\bf m}_\rho $ on $S\M$ and such that the conditional measures on fibers of the projection  map  $\wh{\pi}$  are proportional to the Lebesgue measure on $m$-dimensional frames.   The following is true.

\begin{prop}\label{bigmeasure} The measure $\wh {\bf m}_\rho $ is stationary under $\nu _\rho$,  i.e.,  it satisfies, for any $C^2$ compactly  supported function $f$ on $\OO ^s(S\M)$, 
\[ \int f (\vf(u)) \, d\nu _\rho (\vf)\, d\wh {\bf m}_\rho (u) \; = \; \int f(u) \,  d\wh{\bf m}_\rho (u) .\]
Moreover, the conditional measures $\wh{\bf m}_{\rho, u}^s $of $\wh {\bf m}_\rho $ with respect to the leaves of the $\wh \W^s $ foliation are absolutely continuous with respect to  Lebesgue. \end{prop}
\begin{proof}The stationarity follows from relation (\ref{stationarity}), the stationarity of  $\wt {\bf m}_\rho $ and the fact that the flow preserves the orthogonal group on the fibers (\cite[Lemma 3.1]{CE}).  The leaves of $\wh \W^s$ are made of whole fibers and project on the leaves of $\wt\W^s$. The conditional measures on the leaves of $\wh \W^s$ are given by the extension by Lebesgue on the fibers of the conditional measures on the leaves of $\wt\W^s.$ By Proposition \ref{harmonic-mea} iii), they are therefore absolutely continuous. \end{proof}

Let $\pi : \OO^s (SM) \to SM$ be the quotient of the map $\wh\pi$ by the action of $G$ and let  $\ov\W^s=\{\ov W^s(u)\}_{u\in O^s(SM)}$ denote  the corresponding quotient  foliation of $\wh \W ^s$.  Let  $D^\infty (\OO^s SM) $ be the space of homeomorphisms  of $\OO^s ( SM)$ that preserve the leaves of $\ov\W^s$ and are $C^\infty $-diffeomorphisms along the leaves. We endow $D^\infty (\OO^s SM) $ with the $C^{0,\infty} $ topology: $\vf,\vf' \in D^\infty (\OO^s SM) $ are close if, for all $r >0$, the $r$-germs  of $\vf$ and $\vf'$ are uniformly close  and the $r$-germs  of $\vf^{-1}$ and $(\vf')^{-1}$ are uniformly close.  By Proposition \ref{stochasticflow} iii),  we can consider $\nu_{\rho}$ as a probability measure on $D^\infty (\OO^s SM) $.

We define the measure  $\ov {\bf m}_\rho$ on $\OO^s (SM) $ such that its $G$-invariant extension to $\OO^s (S\M)$ is $\wh{\bf m}_{\rho}$.  We see that  $\ov {\bf m}_\rho$ is a probability measure that projects to  
${\bf m}_\rho $ on $SM$ and is such that the conditional measures on fibers of $\pi$ are proportional to Lebesgue on $m$-dimensional  frames.  As a consequence of Proposition \ref{bigmeasure}, we have

\begin{cor}\label{small-measure} The measure $\ov {\bf m}_\rho $ is stationary under $\nu _\rho$, i.e.,  it satisfies, for any continuous function $f$ on $\OO ^s(SM)$, 
\begin{equation}\label{stationary-bar}\int f (\vf u) \, d\nu _\rho (\vf) \, d\ov {\bf m}_\rho (u) \; = \; \int f(u) \,  d\ov {\bf m}_\rho (u) .\end{equation}
Moreover, the conditional measures $\ov {\bf m}_{\rho, u}^s $ of $\ov {\bf m}_\rho $ with respect to the leaves of the $\ov \W^s $ foliation are absolutely continuous with respect to Lebesgue. \end{cor}

We are interested in the limit measures of ${\bf m}_{\rho}$'s when $\rho$ goes to $-\infty$. Let ${\bf m}$ be such a limit point and let $\ov{\bf m}$ be the probability measure on $\OO^s ( SM)$ that projects to  
${\bf m}$ on $SM$ and is  such that the conditional measures on fibers of $\pi$ are proportional to the  Lebesgue measure  on $m$-dimensional  frames. Then  $\ov{\bf m}$  is the limit of $\ov{\bf m}_{\rho}$ along the same subsequence.   Let $\{\ov {\bf \Phi }_{-t}\}_{t\in \Bbb R}$ be the reversed stable frame flow.  Then  $\ov{\bf m}$ is invariant under $\ov {\bf \Phi }_{-t}$.  To show ${\bf m}$ is Liouville, it suffices to show the conditional measures of $\ov{\bf m}$ on the leaves of $\ov  \W^s$ are absolutely continuous with respect to Lebesgue. But this does not follow from Corollary \ref{small-measure}  by the same reason that we mentioned in Remark \ref{rem-1}. What we are going to  do in the next section  is to analyze the entropy related to the natural  random dynamics for $\ov {\bf m}_\rho$ that arises in the stationarity relation  (\ref{stationary-bar}).

\section{Entropy of random mappings}

 We    consider the action on $\OO^s (SM)$ of the  random elements of  $D^\infty (\OO^s SM) $ with distribution $\nu _\rho , -\infty \leq \rho < 0.$ Namely, let $ \mathcal{S} := (D^\infty (\OO^s SM) )^{\N\cup \{0\}},$ 
 endowed with the product measures $\nu _\rho^{\otimes \N\cup \{0\}}$ (with the convention that $\nu _{-\infty} $ is the Dirac measure at $\ov {\bf \Phi} _{-1}$) and the shift transformation $\s$. On the space $\mathcal{T}:= \mathcal{S} \x \OO^s (SM),$ define the transformation $\tau$ by:
\[ \tau (\underline \vf , {u}) \; := \; (\s \underline \vf , \vf _0 {u}) .\]

For $ -\infty < \rho < 0, $ let $ \ov {\bf m}_\rho $ be the stationary measure from Corollary~\ref{small-measure} and for $ \rho = -\infty, $ let $ \ov {\bf m}_{- \infty } = \ov {\bf m}  $  be some weak* limit of $\ov {\bf m}_\rho $ as $\rho \to - \infty.$ For $-\infty \leq \rho <0 ,$ the measure $\mu _\rho := \nu _\rho^{\otimes \N\cup \{0\}} \otimes \ov{\bf m}_\rho $ is invariant under the transformation $\tau.$  

  Let $\mathcal{P}$ be a measurable partition of $\T$ with finite or countably many elements. We assume 
$-\int \log (\ov{\bf m}_{\rho}(\mathcal{P}))\ d\mu_{\rho}<+\infty$.   For   $n\in \Bbb N$, set $\mathcal{P}_{-1}=\mathcal{P}$
  and $\mathcal{P}_{-n}:=\mathcal{P}\bigvee \tau^{-1}\mathcal{P}\bigvee\cdots\bigvee\tau^{-(n-1)}\mathcal{P}$ for $n>1$, where $\bigvee$ denotes the join of partitions, i.e., the refinement of partitions by taking intersections. For $(\underline\vf, u)\in \T$, let  
$\mathcal{P}_{-n}(\underline\vf, u)$ denote the element of $\mathcal{P}_{-n}$ that contains $(\underline\vf, u)$.  We define the entropy $h_{\rho}^s$ for $\ov{\bf m}_\rho$ as 
\[
h^s_\rho := \sup_{\PP}\underline{h}_{\rho, \mathcal{P}}^s,
\]
where
\begin{equation}\label{underline-h-rho}
\underline{h}_{\rho, \mathcal{P}}^s:=\liminf_{n\to +\infty} -\frac{1}{n} \int \log \ov{\bf m}_{\rho, u}^s \left(\mathcal{P}_{-n}(\underline\vf, u)\right) \, d\mu _\rho (\underline \vf, u).
\end{equation}
For a formal definition of $\ov{\bf m}_{\rho, u}^s$, we  should use a measurable partition $\mathcal{R}$ subordinated to $\ov \W^s$  (see Section 4 for details). But the value of $\underline{h}_{\rho, \mathcal{P}}^s$ does not depend on the choice of   such a  subordinated  partition and is thus well-defined.  Observe that
\[- \int \log \ov{\bf m}_{\rho, u}^s \left(\mathcal{P}_{-n}(\underline\vf, u)\right) \, d\mu _\rho (\underline \vf, u)
 \leq -\int \log \ov{\bf m}_{\rho} \left(\mathcal{P}_{-n}(\underline\vf, u)\right) \, d\mu _\rho (\underline \vf, u).\] Using the random Ruelle  inequality (cf. \cite{BB, Ki2}),  we obtain  that $\underline{h}_{\rho, \mathcal{P}}^s$ is bounded independent of $\mathcal{P}$. Hence $h^s_\rho$ is   finite.  Note also that  $\ov{\bf m}_{\rho,u}^s$ is absolutely continuous with respect to Lebesgue with a smooth density.

For the computation of $\underline{h}_{\rho, \mathcal{P}}^s$, we can restrict the conditional measure $\ov{\bf m}_{\rho,u}^s$ to  the local stable leaf  $\ov {W}^s_{loc, \epsilon} (u):= \{ w \in \ov{W}^s (u) :\  d_{\ov{W}^s}(w,u) < \epsilon\}$ for $\epsilon$ small enough. Recall that  $\vf \in D^{\infty } (\OO^s SM)$ preserves each leaf $\ov W^s (u) $  and is a $C^\infty $ diffeomorphism along it. Write $J(\vf ,u)$ for the Jacobian determinant of  the tangent map  of $\vf \large|_{\ov W^s (u)} $ at $u$.  We will conclude Theorem \ref{mainth} from the following two propositions.
\begin{prop}\label{Thieullen} For $ -\infty < \rho < 0$, 
\begin{equation}\label{equ-Thieullen}
h^s_{\rho}\geq \int  \log J(\vf , u) \, d\nu_\rho (\vf) \ d\ov{\bf m}_\rho (u).
\end{equation}
\end{prop}

\begin{prop}\label{mainprop}Let $\rho_p, p\in \N,$   be a sequence such that  $\rho_p\to -\infty$ and  ${\bf m}_{\rho_p}$ converge  to the probability measure ${\bf m}$ as $p \to +\infty $,  and let ${\ov{\bf m}}$ be as above.  Then
\begin{equation*}\label{main-equ}
h_{\ov{\bf m}}^s:= h^s_{-\infty} \geq \limsup_{p \to +\infty } h_{\rho_p}^s. 
\end{equation*}
\end{prop}

The proofs  of  Proposition \ref{Thieullen}  and Proposition \ref{mainprop} use completely different techniques and will be presented in this section and the following section, respectively. 

In the following, we shall use $H_{\vartheta}(\mathcal{A})$ to denote the entropy of a measurable partition $\mathcal{A}$ with respect to a measure $\vartheta$ of some space and use $H_{\vartheta}(\mathcal{A}|\mathcal{B})$ to denote the entropy of $\mathcal{A}$  conditioned on some measurable partition $\mathcal{B}$,  whenever these entropies  are well-defined.   We shall denote $\ov{m}$  for the dimension of  $\ov W^s $;  for $(\underline \vf , {u})\in \mathcal{T}$, we shall  write  
\[\underline\vf|_0={\rm Id}\ \ {\rm and}\ \ \underline\vf|_n=\vf _{n-1} \circ \cdots \circ \vf _0, \ \forall n\geq 1,\]
and $J(\underline\vf|_n , u)$ for the Jacobian determinant of  the tangent map  of $\underline\vf|_n\large|_{\ov W^s (u)} $ at $u$. Clearly, we have  $J(\underline\vf|_1, u)=J(\vf_0, u)$ for $\underline\vf=(\vf_0,\vf_1,\cdots)\in  \mathcal{S}$.

\begin{proof}[Proof of Theorem \ref{mainth}]
Let $\rho_p, p\in \N$,  be a sequence such that $\rho_p\to -\infty$ and ${\bf{m}}_{\rho_p}$ converge  to the probability measure $\bf m$ as $p\to +\infty$,  and let ${\ov{\bf m}}$ be as above. Recall that $\ov {{\bf\Phi}}_{-1}$ is the time one map of the reversed frame  flow on $\OO^s (SM)$ which is a compact isometric extension of the time one map of the reversed geodesic flow ${\bf\Phi}_{-1}$ on $SM.$  Hence, 
\[
h_{\bf m}=h_{\ov{\bf m}}. 
\]

On the other hand, we have:
\[ h_{\ov {\bf m}} \;=\; \sup_{\PP} \lim_{n\to +\infty} \frac{1}{n} H_{\ov{\bf m}} (\PP _{-n}) \;\geq \; \sup_{\PP} \liminf_{n\to +\infty} -\frac{1}{n} \int \log \ov{\bf m}_{u}^s \left(\mathcal{P}_{-n}( u)\right) \, d\ov{\bf m}(u)\;=\;h^s_{\ov {\bf  m}}.\]
Assume Proposition \ref{Thieullen} and Proposition  \ref{mainprop} hold true. Then, 
\[
h_{\ov{\bf m}}^s\geq \limsup\limits_{p\to +\infty}\int  \log J(\vf , u) \, d\nu_\rho (\vf)\,  d\ov{\bf m}_{\rho_p} (u)=\int  \log J(\ov{\bf\Phi}_{-1}, u) \, d\ov{\bf m}(u),
\]
where the last equality holds by Proposition \ref{con-geo-flow} iii). Altogether, we obtain
\[
h_{\bf m}\geq \int  \log J(\ov{\bf\Phi}_{-1}, u) \, d\ov{\bf m}(u). 
\]
Note that $\ov{\W}^s$ is the central unstable foliation for $\ov{\bf\Phi}_{-1}$, so that  $\int  \log J(\ov{\bf\Phi}_{-1}, u) \, d\ov{\bf m}(u)$ is the integral of the sum of the nonnegative exponents of $\ov {{\bf\Phi}}_{-1}$ for $\ov{\bf m}$; neither the direction of the flow nor the vertical directions tangent to the fibers provide positive exponents,   so that $\int  \log J(\ov{\bf\Phi}_{-1}, u) \, d\ov{\bf m}(u)$ is the integral of the sum of the positive  exponents of ${\bf\Phi}_{-1}$ for $\bf m$.   By \cite {BR}, $\bf m $ is the normalized  Liouville  measure. 
\end{proof}

\subsection {Proof of Proposition \ref{Thieullen}}
The deterministic diffeomorphism version estimation of (\ref{equ-Thieullen}) is standard using Pesin theory (cf. \cite{Mane}).  But this cannot be used directly  since we are in the random and non-invertible case.

Clearly,  Proposition  \ref{Thieullen}  would follow  if we can show the sample measures are SRB. This approach might work  since in a  similar context, Blumenthal-Young (\cite{BY}) showed the sample measures are SRB. We didn't try that way since we don't  need that strong conclusion and the intuition for  Proposition \ref{Thieullen} is relatively simpler.

For  a non-invertible  endomorphism of a compact manifold preserving an absolutely continuous measure, the corresponding measure theoretical  entropy is at least the integral of the logarithm of the Jacobian, which coincides with  the so-called folding entropy (cf. \cite{R}, \cite{LiS11}).  Proposition \ref{Thieullen} is intuitively a random conditional  version of this phenomenon. But it might be subtle since we are considering the conditional measures and are in the random case.  So we will give some details for the key steps.

 We  first  recall some notations and results concerning Pesin  local Lyapunov charts theory for  random diffeomorphisms. In many places, we have to take invariant variables  instead of constants since our system $(\mathcal{T}, \tau,  \mu _\rho)$  is invariant, but not necessarily ergodic in general.

\begin{lem}\label{lem-Os}(\cite{Os}) For each $\rho<0$,  there is a measurable ${\bf \Om}\subset \T$ with $\mu_{\rho}({\bf \Om})=1$ such that for $(\underline\vf, u)\in {\bf \Om},$ there exist $r(\underline\vf, u)\in \Bbb N$ and, for $i, 1\leq i\leq r(\underline\vf, u)$,  $\chi_{i}(\underline\vf, u), d_i(\underline\vf, u)$ and a filtration 
\[
\{0\}=V_{r(\underline\vf, u)+1}\subset V_{r(\underline\vf, u)}\subset\cdots \subset V_1=T_u\ov{W}^s(u)
\]
with the following properties:
\begin{itemize}
\item[i)] all of $r, \chi_i, d_i, V_i$'s  depend measurably on $(\underline\vf, u)$;
\item[ii)]$\lim_{n\to +\infty}\frac{1}{n}\log \|D_u(\underline\vf|_n) (e) \|=\chi_i(\underline\vf, u)$ for $e\in V_{i}(\underline\vf, u)\backslash V_{i+1}(\underline\vf, u)$; 
\item[iii)]$d_i(\underline\vf, u)={\rm dim} V_i(\underline\vf, u)-{\rm dim}V_{i+1}(\underline\vf, u)$ and  $\sum_{i=1}^{r(\underline\vf, u)} d_i(\underline\vf, u)=\ov{m}$;
\item[iv)]$\int \sum_{i=1}^{r(\underline\vf, u)} \chi_i(\underline\vf, u) d_i(\underline\vf, u)\ d\mu_{\rho}(\underline\vf, u)=\int \log J(\vf, u)\ d\nu_{\rho}(\vf)\ d\ov{\bf m}_{\rho}(u)$. 
\end{itemize}
\end{lem}

\begin{lem}\label{lem-LQ}(cf. \cite[Chapter III, Section 1]{LQ})  For each $\rho<0$, given a small enough positive $\tau$-invariant function $\epsilon$ on $\T$, there is a positive function $\kappa$ on ${\bf \Omega}\times \{\N\cup \{0\}\}$  such that for $n\in \N\cup \{0\}$, 
\begin{equation}\label{cite-kappa}
\kappa\big((\underline\vf, u), n+1\big)\leq e^{\epsilon}\cdot \kappa\big((\underline\vf, u), n\big),
\end{equation}
a positive constant $\kappa_0$ 
and a sequence  Euclidean metrics $\|\cdot\|'_{(\underline\vf, u), n}$ on $T_{\underline\vf|_n(u)}\ov{W}^s(\underline\vf|_n(u))$ such that  for all  $n\in \N\cup \{0\}$,
\begin{itemize}
\item[i)]$\kappa_0\|\cdot\|_{\underline\vf|_n(u)}\leq \|\cdot\|'_{(\underline\vf, u), n}\leq  \kappa((\underline\vf, u), n)\|\cdot\|_{\underline\vf|_n(u)}$, where $\|\cdot\|_{\underline\vf|_n(u)}$ is the Riemannian norm on $T_{\underline\vf|_n(u)}\ov{W}^s(\underline\vf|_n(u))$;
\item[ii)]${\bf F}_{(\underline\vf, u), n}(e):=\exp_{\underline\vf|_{n+1}(u)}^{-1}\circ \vf_n\circ \exp_{\underline\vf|_n(u)}(e)$ is defined for $e$ with $\|e\|'_{(\underline\vf, u), n}\leq \epsilon(\underline\vf, u)$;
\item[iii)]${\bf F}_{(\underline\vf, u), n}$ is $C^2$ and  $\|D^{(2)}{\bf F}_{(\underline\vf, u), n}\|'_{(\underline\vf, u), n, n+1}<\kappa \big((\underline\vf, u), n\big)$,  where by  ${\|\cdot\|}'_{(\underline\vf, u), n, n+1}$ we mean the norm of the tangent map calculated using $\|\cdot\|'_{(\underline\vf, u), n}$ and $\|\cdot\|'_{(\underline\vf, u), n+1}$;
\item[iv)] the map $D{\bf F}_{(\underline\vf, u), n} $ satisfies 
\[ \| D_e{\bf F}_{(\underline\vf, u), n} - D_0{\bf F}_{(\underline\vf, u), n}\|'_{(\underline \vf, u), n, n+1} \leq \epsilon (\underline\vf, u)  \|e\|'_{(\underline\vf, u), n};\]
\item[v)] the map $D_0{\bf F}_{(\underline\vf, u), n} $ satisfies
 $$ e^{\chi_i(\underline\vf, u)-\epsilon(\underline\vf, u)} \| e \|'_{(\underline\vf, u), n} \leq \| D_0{\bf F}_{(\underline\vf, u), n} (e)\|'_{(\underline\vf, u), n+1} \leq e^{\chi_i(\underline\vf, u)+\epsilon(\underline\vf, u)} \|e\|'_{(\underline\vf, u), n},$$ for all  $e \in E_i(\tau^n(\underline\vf, u))$.  Moreover, for $i, 1\leq i\leq r(\underline\vf, u)$, the spaces $E_j(\underline\vf, u)$, $j\geq i$, generate $V_i(\underline\vf, u)$.
\end{itemize}
\end{lem}
(Since elements of $D^\infty (\OO^s SM) $ preserve the leaves of $\ov\W^s$,  Lemma \ref{lem-LQ} can be obtained as in \cite[Chapter III, Section 1]{LQ}) using the natural auxiliary charts $E_j(\underline\vf, u)'s$ and Lemma \ref{lem-Os}.)

For Proposition \ref{Thieullen},  we shall follow   Ma\~{n}\'{e} (\cite{Mane}) to give a local version  of (\ref{equ-Thieullen}) for  Bowen balls defined using the norms in Lemma \ref{lem-LQ} and then compare it with local entropy for special partitions.  Note that we are in the non-invertible case,  $\kappa$ is not invariant (i.e.,  for $(\underline\vf, u)\in {\bf \Omega}$,  $\kappa ((\underline\vf, u), n)$ does not equal to $\kappa (\tau^n(\underline\vf, u), 0)$ in general).  To overcome this  deficiency, we will pick up a set ${\rm A}$ with measure  close to 1   and define   modified  Bowen balls  associated to ${\rm A}$.

Let $\epsilon_0>0$,  $-\infty<\rho<0$ be fixed. Choose $\epsilon<\epsilon_0$ as in Lemma \ref{lem-LQ}.  For any $a\in (0, 1)$, we choose   a measurable set ${\rm A}\subset {\bf \Omega}$ with $\mu _\rho ({\rm A}) > 1-a $ as follows. 
By the ergodicity of $\nu _\rho^{\otimes \N\cup \{0\}}$ with respect to $\sigma$ and the integrability property (\ref{r-derivatives}),  for $\nu _\rho^{\otimes \N\cup \{0\}}$ almost every $\underline{\vf}$, 
\[
\lim\limits_{n\to +\infty}\frac{1}{n}\sum_{i=1}^{n}\big(\log \max _u \| \vf_i|_{\wh W^s (u)} \|_{C^2}\big)^+=\int \big(\log \max _u \| \vf|_{\wh W^s (u)} \|_{C^2}\big)^+\ d\nu_{\rho}(\vf)=:L<+\infty. 
\]
For any $b>0$, let 
\[
{\rm A}(b):=\left\{\underline{\vf}\in \S:\ \prod_{i=0}^{n-1}\max _u \| \vf_i|_{\wh W^s (u)}\|_{C^2}\leq be^{2Ln}\ \mbox{for all }\ n\in \Bbb N\cup\{0\}\right\}. 
\]
Then there exists $b>0$ large such that 
\begin{equation}\label{choice-b}
\nu _\rho^{\otimes \N\cup \{0\}}({\rm A}(b))>1-\frac{1}{2}a. 
\end{equation}
Let $b$ be as in (\ref{choice-b}).  For $l>0$, set 
\begin{equation}\label{set-rmA-b-l}
{\rm A}(b, l):=\big\{(\underline{\vf}, u)\in {\bf \Omega}:\ \underline{\vf}\in {\rm A}(b), \ \eta\kappa_0\big(\kappa\big((\underline\vf, u), 0\big)\big)^{-2}>l\big\}. 
\end{equation}
Choosing  $\eta>0, \epsilon(\underline{\vf}, u)$ with   $\epsilon(\underline{\vf}, u)<\epsilon_0$
 and $l$  to be small enough,  we can obtain a measurable set 
\begin{equation}\label{theset-A}
 {\rm A}:={\rm A}(b, l) \cap \{ (\underline \vf, u): \eta <\epsilon(\underline{\vf}, u) \}
 \end{equation} 
with  $\mu_{\rho}$ measure greater than $1-a$. 

Let  ${\rm A}$ be as in (\ref{theset-A}).  For  $\mu_{\rho}$ almost all $(\underline{\vf}, u)\in {\rm A}$, it  will return to ${\rm A}$ under the iterations of the  map $\tau$ for infinitely many times. Hence, for any such  $(\underline{\vf}, u)$ and $k\geq 1$, 
 \begin{equation}\label{new-norm-1} \| .\|''_{(\underline \vf, u ), k} := \| .\|'_{\tau^{{\mathtt  N}_k^{\rm A}}(\underline \vf, u), k-{\mathtt N}_k^{\rm A}}
\end{equation}
is well-defined, 
where ${\mathtt N}_k^{\rm A}$ is the last non-negative time before  or  equal to $k$ with $\tau^{{\mathtt N}_k^{\rm A}}(\underline \vf, u) \in {\rm A}$.   For  $\eta >0$, $(\underline{\vf}, u)\in {\rm A}$   such that $\eta<\epsilon(\underline\vf, u),$  and $ n \in \N$, let us define the \emph{modified random $\ov \W^s$-Bowen  ball (with respect to ${\rm A}$)}  by 
\begin{equation}\label{mod-Bowen-ball-1}
{\begin{split}
& {\bf B}^s_{\rm A}(\underline{\vf}, u, \eta, n) :=\left\{e\in T_u\ov{W}^s(u):\  \|e\|''_{(\underline\vf, u), 0}<\eta \kappa^{-1}\big((\underline\vf, u), 0\big) , \ \mbox{and for}\  k, \  1\leq k\leq n, \right.\\
&\ \ \ \ \ \ \ \ \ \ \ \  \ \ \ \ \ \ \ \ \ \ \ \ \ \ \ \ \ \ \ \ \ \ \ \ \ \ \ \ \ \ \ \   \left.\big\|{\bf F}_{({\un \vf },u)}\big|_k(e)\big\|''_{(\underline\vf, u), k}<\eta \kappa^{-1}\big(\tau^{{\mathtt N}_k^{\rm A}}(\underline\vf, u), k-{\mathtt N}^{\rm A}_k\big) \right\},
\end{split}}
\end{equation} where ${\bf F}_{({\un \vf },u)}\big|_k:={\bf F}_{(\underline\vf, u), k-1}\circ  \cdots\circ {\bf F}_{(\underline\vf, u), 0}$. 

The following can be considered as a  first step coarse local  version of Proposition \ref{Thieullen}: 

\begin{lem}\label{follow-Mane-1} Let $-\infty<\rho<0,$ $a\in (0, 1)$,   $\epsilon_0>0$,   and $\eta >0$ be fixed.  Choose $\epsilon<\epsilon_0$ as in Lemma \ref{lem-LQ}. Let  ${\rm A}\subset {\bf \Omega}$  be as in (\ref{theset-A}).  Then, there is a positive geometric constant   ${\bf C}_0$   such that, setting   ${\bf c}_0 :=\frac{1}{2}  {\ov m}  \log \frac{l\kappa _0 }{\eta} , $ 
 for $\mu_{\rho}$ almost all $(\underline\vf, u)$ and all $n\in \Bbb N$,   
\begin{equation}\label{Fo-Mane}
-\frac{1}{n} \log \ov{\bf m}_{\rho,u}^s \left(\exp_u({\bf B}^s_{\rm A}(\underline \vf,  u,  \eta, n ))\right)\geq -\frac{1}{n}\log J(\un \vf |_n, u)-3\ov{m}\epsilon_0 -\frac{1}{n} {\bf C}_0+{\bf c}_0.
\end{equation}
\end{lem}

    \begin{proof} The set ${\bf B}^s_{\rm A}(\underline{\vf}, u, \eta, n)$ is  empty if   $(\underline\vf, u) \not \in {\rm A}.$   Otherwise, by definition and Lemma \ref{lem-LQ} iv), ${\bf B}^s_{\rm A}(\underline \vf,  u,  \eta, n )$ is contained in the set of vectors $ e\in T_u\ov{W}^s(u)$ such that  
    \begin{equation}\label{B-A-s-set}
\begin{split}& \ \  \|e\|'_{(\un\vf, u), 0}<\eta \kappa^{-1}\big((\un\vf, u), 0\big),\ \big\| {\bf F}_{({\un \vf },u)}\big|_n(e) \big\|''_{(\un\vf, u), n}<\eta \kappa^{-1}\big(\tau^{{\mathtt N}_n^{\rm A}}(\underline\vf, u), n-{\mathtt N}^{\rm A}_n\big)\; {\textrm {  and }} \\
&\ \ \ \ \ \ \ \ \ \ \ \  \ \ \ \ \ \ \  \big|{\textrm {Det}}' D_e {\bf F}_{({\un \vf },u), k}\big| \geq \big|{\textrm {Det}}' D_0 {\bf F}_{({\un \vf },u), k}\big| (1-\epsilon(\un \vf ,u))^{\ov{m}} \ \mbox{for}\ 1\leq k\leq n,
\end{split}
\end{equation}
where ${\textrm {Det}}'$ is the determinant of a linear mapping in the metrics $\| \cdot \|'_{\tau^{{\mathtt N}_k^{\rm A}}(\underline\vf, u), k-{\mathtt N}^{\rm A}_k}$ and $\| \cdot \|'_{\tau^{{\mathtt N}_k^{\rm A}}(\underline\vf, u), k+1-{\mathtt N}^{\rm A}_k}$.  By construction, $\tau^{{\mathtt N}_k^{\rm A}}(\underline\vf, u) \in {\rm A}$ and   $ \kappa\big(\tau^{{\mathtt N}_k^{\rm A}}(\underline\vf, u), 0\big)    \leq \big(\frac {\eta}{l} \kappa _0 \big)^{1/2} .$  Assume $\tau^{k+1}(\underline \vf, u) \in {\rm A}$, i.e., ${\mathtt N}_{k+1}^{\rm A}=k+1$. Then 
\begin{equation}\label{ch-ch-norm}
\begin{split}
\| \cdot \|'_{\tau^{{\mathtt N}_k^{\rm A}}(\underline\vf, u), k+1-{\mathtt N}^{\rm A}_k}&\leq   e^{({\mathtt N}_{k+1}^{\rm A} - {\mathtt N}_{k}^{\rm A}) \epsilon }  \kappa\big(\tau^{{\mathtt N}_k^{\rm A}}(\underline\vf, u), 0\big)     \kappa_0^{-1} \|\cdot\|'_{\tau^{k+1}(\underline\vf, u), 0}\\
&\leq    e^{({\mathtt N}_{k+1}^{\rm A} - {\mathtt N}_{k}^{\rm A}) \epsilon }  \big(\frac {\eta}{l\kappa_0} \big)^{\frac{1}{2}}  \|\cdot\|'_{\tau^{k+1}(\underline\vf, u), 0}.
\end{split}
\end{equation}
By chopping    $\un \vf |_n(u)$  into pieces in between returning times and using (\ref{B-A-s-set}), (\ref{ch-ch-norm}),  we see that   there exists a geometric constant $C$ such that the set $\exp_u {\bf B}^s_{\rm A}(\underline \vf,  u,  \eta, n )$ is contained in the set $\underline{{\bf B}}^s_{\rm A}(\underline \vf,  u,  \eta, n )$ of points  $ w\in \ov{W}^s(u)$ such that 
\begin{align*}& d_{\ov W^s(u)} (w,u) <C \kappa _0^{-1} \eta \kappa^{-1}(\un\vf, u),\  d_{\ov W^s(\un\vf |_n u)} (\un\vf |_n w,\un\vf |_n u) <C \kappa_0^{-1} \eta \kappa^{-1}\big(\tau^{{\mathtt N}_n^{\rm A}}(\underline\vf, u), n-{\mathtt N}^{\rm A}_n\big) \\
&\ \ \ {\textrm {and }}  \ J(\un \vf |_n, w)\geq (C^{-1} \kappa _0)^{\ov{m}}(1-\epsilon(\un \vf ,u))^{n\ov{m}}J(\un \vf |_n, u) e^{-n \ov m\epsilon} \big(\frac {l \kappa _0}{\eta} \big)^{\frac{1}{2} \ov m \# \{{\mathtt N}_k^{\rm A}(\underline\vf, u):\  k\leq n\}}.
\end{align*}
It follows that, denoting $\overline{\lambda}_{u}^s$ the Lebesgue measure on $\ov{W}^s(u)$, 
\[
\ov{\bf m}_{\rho, u}^s \left(\exp_u({\bf B}_{\rm A}^s(\underline \vf,  u,  \eta, n ))\right)\leq \frac{e^{C'}}{\overline{\lambda}_{u}^s(\mathcal{R}(\un \vf |_n (u)))}\int _{ \un \vf |_n \left( \underline{{\bf B}}_{\rm A}^s(\underline \vf,  u,  \eta, n )\right)} J^{-1}(\un \vf |_n, w)\, d {\ov{\lambda}}_{\un \vf |_n u}^s(w),
\]
where $C'$ is a positive constant taking into account  the regularity of the density for a fixed $\rho$.  It follows that, with our definition of $ {\bf c}_0, $
\begin{equation*}
\begin{split} &\ov{\bf m}_{\rho, u}^s \left(\exp_u({\bf B}_{\rm A}^s(\underline \vf,  u,  \eta, n ))\right)\\
&\ \ \leq\frac{e^{C'}}{\overline{\lambda}_{u}^s(\mathcal{R}(\un \vf |_n (u)))}  J^{-1}(\un \vf |_n, u)(C \kappa _0^{-1})^{\ov{m}}(1-\epsilon(\un \vf ,u))^{-n\ov{m}}e^{n\overline{m}\epsilon} e ^{- {\bf c}_0\# \{{\mathtt N}_k^{\rm A}(\underline\vf, u):\  k\leq n\}}. \end{split}\end{equation*}
Note that  the partition $\mathcal{R}$ is such that each element  contains  a ball with radius greater than some positive constant (see Section 4), 
we obtain  some constant ${\bf C}_0>0$ such that 
\[ -\frac{1}{n} \log \ov{\bf m}_{\rho,u}^s \left(\exp_u({\bf B}_{\rm A}^s(\underline \vf,  u,  \eta, n ))\right)\geq \frac{1}{n} \log J(\un \vf |_n, u) +\ov{m}\log(1-\epsilon(\un \vf , u))-\overline{m}\epsilon- \frac{1}{n} {\bf C}_0+ {\bf c}_0.\]
The estimation in (\ref{Fo-Mane}) follows for $\epsilon_0$ small enough. \end{proof}

By  (\ref{set-rmA-b-l}) and Lemma \ref{lem-LQ} i),  we have 
\[
 {\bf c}_0 =\frac{1}{2} \ov m \log \frac {l \kappa _0}{\eta }\leq  \ov m \log \frac{\k_0}{ \min _{( \underline{\vf}, u) \in {\rm A} } \kappa \big( (\underline{\vf}, u), 0\big)}\leq 0 .\]
So the   estimation in   (\ref{Fo-Mane}) might be  too    coarse since  ${\bf c}_0$  is not a priori small compared with $\epsilon_0$.  But  ${\bf c}_0$ remains unchanged  when we consider (\ref{Fo-Mane}) for  Bowen balls for any power of $\tau$, hence  it will not enter the  lower bound  estimation of entropy in (\ref{equ-Thieullen}).

More precisely, let $\mathsf M\in \Bbb N$ be fixed. For $\underline \vf\in \mathcal S$, write 
\[ \underline \vf':=(\vf'_0, \vf'_1,\cdots, \vf'_k, \cdots), \ \mbox{where}\  \vf'_k := \underline{\vf} \circ \sigma ^{k\mathsf M}|_\mathsf M.\]
 Let 
\[\sigma'(\underline \vf'):=(\vf'_1,\cdots, \vf'_k, \cdots) \  \mbox{and}\ \tau' (\underline \vf' , {u}) \; := \; (\s' (\underline \vf), \vf' _0 {u}).\]
The transformation   $\tau'$ can be identified with $\tau^{\mathsf M}$.   We can use the same $\kappa _0 , \kappa $ for $\tau'$ as for $\tau$, but now  $\epsilon$ in (\ref{cite-kappa}) has to be changed into $\mathsf M \epsilon$.   So we have to choose $\epsilon '$ so that  $\mathsf M \epsilon' < \epsilon_0 .$ Choose $\eta' < \eta $ small enough that, if  $l' : = l \frac{\eta '}{\eta}$, the measurable set 
\begin{equation}\label{theset-A'}
 {\rm A'}:={\rm A}(b, l') \cap \{ (\underline \vf, u): \eta '<\epsilon '(\underline{\vf}, u) \}
 \end{equation} 
has  $\mu_{\rho}$ measure greater than $1-a$.

For  $\mu_{\rho}$ almost all $(\underline{\vf}, u)=(\underline{\vf}', u)\in {\rm A'}$ and $k\in \Bbb N$, let $\mathsf M^{\rm A'}_k$ denote  the last non-negative time before  or  equal to $k$ such that $(\tau')^{ \mathsf M^{\rm A'}_k}(\underline{\vf}', u)\in {\rm A'}$.  Similar to (\ref{new-norm-1}) and (\ref{mod-Bowen-ball-1}),   we define 
\[ \| .\|^{'', \mathsf M}_{(\underline \vf', u ), k} := \| .\|'_{(\tau')^{{\mathsf M}_k^{\rm A'}}(\underline \vf', u), (k-{\mathsf M}_k^{\rm A'})\mathsf M}, 
\] 
and for $\eta' >0$, $(\underline{\vf}, u)\in {\rm A'}$   such that $\eta '<\epsilon '(\underline\vf, u),$ and $ n \in \N$,  we define the \emph{modified random $\ov \W^s$-Bowen  ball  for $\tau'$ (with respect to ${\rm A'}$) } by 
\begin{align*}
{\bf B}^{s, \mathsf M}_{\rm A'}(\underline{\vf}', u, \eta, n) &:=\left\{e\in T_u\ov{W}^s(u):\  \|e\|''_{(\underline\vf', u), 0}<\eta '\kappa^{-1}\big((\underline\vf', u), 0\big) , \ \mbox{and for}\  k, \  1\leq k\leq n, \right.\\
&\ \ \ \ \ \ \ \ \ \ \ \  \ \ \ \ \ \ \ \ \ \left.\big\|{\bf F}_{({\un \vf },u)}\big|_{k\mathsf M}(e)\big\|^{'', \mathsf M}_{(\underline\vf', u), k}<\eta \kappa^{-1}\big((\tau')^{{\mathsf M}_k^{\rm A'}}(\underline\vf', u), (k-{\mathsf M}^{\rm A'}_k)\mathsf M\big) \right\}.
\end{align*}
Then following the argument in Lemma \ref{follow-Mane-1}, we obtain  (observe that, by our choice of $l'$, $ {\bf c}_0 =\frac{1}{2} \ov m \log \frac {l \kappa _0}{\eta }= \frac{1}{2} \ov m \log \frac {l'\kappa _0}{\eta'}$ has the same value as in Lemma \ref{follow-Mane-1})
\begin{lem}
 Let $-\infty<\rho<0,$ $a\in (0, 1)$,  $\mathsf M\in \Bbb N$  and $\epsilon_0$ be fixed. Let  $  \eta', \epsilon', {\rm A'}$  be as in (\ref{theset-A'}) and let $  {\bf c}_0, {\bf C}_0 $ be as in Lemma \ref{follow-Mane-1}. Then,  for $\mu_{\rho}$ almost all $(\underline\vf, u)\in {\rm A'}$ and all $n\in \Bbb N$,   
\begin{equation}\label{Fo-Mane-mathsfM}
-\frac{1}{n} \log \ov{\bf m}_{\rho,u}^s \left(\exp_u({\bf B}^{s, \mathsf M}_{\rm A'}(\underline \vf',  u,  \eta, n ))\right)\geq -\frac{1}{n}\log J(\un \vf |_{n\mathsf M}, u)-3\ov{m}\mathsf M\epsilon_0 -\frac{1}{n} {\bf C}_0+{\bf c}_0.
\end{equation} 
\end{lem}

Following  Ma\~{n}\'{e} (\cite{Mane}) (see also \cite{Th}), we can proceed  to  find partitions which have local entropy lower bound as in (\ref{Fo-Mane-mathsfM}) in our non-invertible random setting. 

\begin{lem}\label{follow-Mane-2}Let $\epsilon_0>0$,  $-\infty<\rho<0$,   $\mathsf M\in \Bbb N$ be fixed.  Let  $  \eta', \epsilon', {\rm A'}$  be as in (\ref{theset-A'}) and let $  {\bf c}_0, {\bf C}_0 $ be as in Lemma \ref{follow-Mane-1}.  There exists  a countable  partition $\mathcal{Q}$ of $\T$ with $-\int \log (\ov{\bf m}_{\rho}(\mathcal{Q}))\ d\mu_{\rho}<+\infty$  such that for $\mu_{\rho}$ almost all $(\underline \vf, u) \in {\rm A'}$, we have  $0<\eta '< \epsilon '(\underline{\vf}, u)<\epsilon_0 /\mathsf M$  and 
\begin{equation}\label{Q-partition}
\mathcal{Q}_{\mathsf M, -n}(\underline\vf, u)\subset \exp_u\big({\bf B}^{s, \mathsf{M}}_{\rm A'}(\underline \vf',  u,  \eta', n )\big),
\end{equation}
where $\mathcal{Q}_{\mathsf M, -n}:=\mathcal{Q}\bigvee (\tau^{\mathsf M})^{-1}\mathcal{Q}\bigvee\cdots\bigvee(\tau^{\mathsf M})^{-(n-1)}\mathcal{Q}$. Consequently, for $\mu_{\rho}$ almost all $(\underline \vf, u) \in {\rm A'}$, 
\begin{equation}\label{loc-ent-parti}
\liminf_{n\to +\infty}-\frac{1}{n}\log \ov{\bf m}_{\rho,u}^s \left(\mathcal{Q}_{\mathsf M, -n}(\underline\vf, u)\right)\geq -\frac{1}{n}\log J(\un \vf |_{n\mathsf M}, u)-3\ov{m}\mathsf M\epsilon_0 -\frac{1}{n} {\bf C}_0+{\bf c}_0.
\end{equation}
\end{lem}
\begin{proof} Clearly, (\ref{loc-ent-parti}) is a consequence of (\ref{Q-partition}) and (\ref{Fo-Mane-mathsfM}). Hence, it suffices to show (\ref{Q-partition}). Let  $  \eta', \epsilon', {\rm A'}$  be as in (\ref{theset-A'}). 
  For  $(\underline{\vf}, u)\in {\rm A'}$,  $\kappa_1>0$ and $n\in \Bbb N$,  set
\begin{equation*}
\begin{split}
&B^{s,\mathsf M, \kappa_1, \kappa}_{\rm A'}(\underline{\vf}',  u,  \eta', n ):=\\
&\!\left\{ {w} \in \ov W^s(u) :  \; d\left(\underline\vf|_{k\mathsf M} (w),  \underline\vf|_{k\mathsf M} (u)\right) < \eta' \kappa_1\!\left(\kappa\big((\tau')^{{\mathsf M}_k^{\rm A'}}(\underline\vf', u), (k-{\mathsf M}^{\rm A'}_k)\mathsf M\big)\right)^{-2},  \forall  0 \leq k \leq n\right\}.
\end{split}
\end{equation*}
By Lemma \ref{lem-LQ} i), we see that there exists  some constant $\kappa_1$ depending on the geometry of $(M, g)$ such that,   for almost all   $(\underline{\vf}, u)\in {\rm A'}$ and  all  $n\in \Bbb N\cup\{0\}$, 
\begin{equation*}
B^{s,\mathsf M, \kappa_1, \kappa}_{\rm A'}(\underline{\vf}',  u,  \eta', n )\subset \exp_u\big({\bf B}^{s, \mathsf{M}}_{\rm A'}(\underline \vf',  u,  \eta', n )\big).
\end{equation*}
Hence,  to find a countable partition $\mathcal{Q}$ satisfying (\ref{Q-partition}), it suffices to find a $\mathcal{Q}$ such that 
\[
\mathcal{Q}_{\mathsf M, -n}(\underline\vf, u)\subset B^{s,\mathsf M, \kappa_1, \kappa}_{\rm A'}(\underline{\vf}',  u,  \eta', n ). 
\]
For 
each $n\in \Bbb N\cup\{0\}$, let  ${\rm A'}_n\subset {\rm A'}$ be the collection of points with $n$ as the first return time to ${\rm A'}$ with respect to  the map $\tau^{\mathsf M}$.  Recall that the local   stable leaf  $\ov {W}^s_{loc, \epsilon_0} (u)= \{ w \in \ov{W}^s (u) :\  d_{\ov{W}^s}(w,u) < \epsilon_0\}$ depends continuously on $u$ and for each $n$, we can choose in a continuous way a  maximal $(4(l')^2b)^{-1}e^{-2(L+\epsilon_0) n\mathsf M}$  separated set in $\ov {W}^s_{loc, \epsilon_0} (u)$.  The cardinality $\mathcal{C}_n$ of such sets satisfies $\mathcal{C}_n \leq K^{n\mathsf M}$ for some $K$. Using these points, we can further slice ${\rm A'}_n$  into $\{{\rm A'}_{n,\ell}\}_{\ell \leq \mathcal{C}_n}$ such that for all $(\underline{\vf}, u)\in {\rm A'}_{n, \ell}$,  the intersection   $\{w:\ (\underline{\vf}, w)\in {\rm A'}_{n, \ell}\}\cap \ov {W}^s_{loc, \epsilon_0} ( u)$   has diameter less than $(2(l')^2b)^{-1}e^{-2(L+\epsilon_0)n\mathsf M}$.
The partition $\mathcal{Q}$ can be chosen to be \[
\{{\rm A'}_{n,\ell}, n\in \Bbb N\cup\{0\}, \ell\leq {\mathcal{C}}_n,  \T\backslash {\rm A'}\}. 
\]
Following \cite{Mane}, one  checks that $\mathcal{Q}$ satisfies $-\int \log (\ov{\bf m}_{\rho}(\mathcal{Q}))\: d\mu_{\rho}<+\infty$  and (\ref{Q-partition}).
\end{proof}

\begin{proof}[Proof of Proposition \ref{Thieullen}] Let $-\infty<\rho<0$ be fixed.  In the following, we show,   for every $\epsilon_0>0$, there exists  a finite measurable partition $\mathcal{P}$ of $\T$ satisfying 
\begin{equation}\label{Mane-2-inequ}\underline{h}_{\rho, \mathcal{P}}^s\geq \int \log J(\vf , u) \, d\nu_\rho(\vf)\  d\ov{\bf m}_\rho (u)-5\ov{m}\epsilon_0.  \end{equation}
Then, by definition of $h_{\rho}^s$ and (\ref{Mane-2-inequ}), 
\[
h^s_{\rho}\geq \underline{h}_{\rho, \mathcal{P}}^s\geq \int \log J(\vf , u) \, d\nu_\rho(\vf)\  d\ov{\bf m}_\rho (u)-5\ov{m}\epsilon_0.
\]
This concludes the proof of Proposition \ref{Thieullen} since  $\epsilon_0$ is arbitrary.

Let $\mathsf M$ be such that $ |{\bf c}_0|<\mathsf M\epsilon_0$. 
Let $a>0$ be small and let ${\rm A'}$ and $\mathcal{Q}$ be as in Lemma \ref{follow-Mane-2}.  Then for $\mu_{\rho}$ almost all $(\underline{\vf}, u)\in {\rm A'}$, (\ref{loc-ent-parti}) holds true. 
Set 
\[\underline{h}_{\rho, \mathcal{Q}}^{s, \mathsf M}:=\liminf_{n\to +\infty} -\frac{1}{n} \int \log \ov{\bf m}_{\rho,u}^s \left(\mathcal{Q}_{\mathsf M, -n}(\underline\vf, u)\right) \, d\mu _\rho (\underline \vf, u).\]
For any $\alpha>0$, by our choice of  $\mathcal R$ in Section 4, it is true that (see Proposition \ref{new-prop4.3})
\[{\underline{h}}^s_{\rho, \mathcal Q}\geq \frac{1}{\mathsf M}\underline{h}_{\rho, \mathcal{Q}}^{s, \mathsf M}-\alpha.\]
Hence, by Fatou Lemma, 
\[
{\underline{h}}^s_{\rho, \mathcal Q}\geq  \int _{\rm{A}}\liminf_{n\to +\infty}-\frac{1}{n\mathsf M}\log \ov{\bf m}_{\rho,u}^s \left(\mathcal{Q}_{\mathsf M, -n}(\underline\vf, u)\right) \ d\mu_{\rho}(\underline\vf, u)-\alpha. 
\] 
Since the function $\log J(\vf,u) $ is integrable  and  $ |{\bf c}_0|<\mathsf M\epsilon_0$,   by  using(\ref{loc-ent-parti}), we obtain, for  $a, \alpha>0$  small,     
\[\underline{h}_{\rho, \mathcal{Q}}^s\geq \int \log J(\vf , u) \, d\nu_\rho(\vf)\  d\ov{\bf m}_\rho (u)-4\ov{m}\epsilon_0.\]
Note that  $\mathcal{Q}$  is such that $-\int \log (\ov{\bf m}_{\rho}(\mathcal{Q}))\ d\mu_{\rho}<+\infty$ 
  and for any finite partition  $\mathcal{P}$  such that $\mathcal{Q}$ is finer than it, 
\begin{eqnarray*}
\underline{h}_{\rho, \mathcal{Q}}^s-\underline{h}_{\rho, \mathcal{P}}^s &\leq & \limsup_{n\to +\infty}  \frac{1}{n}  \int H_{\ov{\bf m}^s_{\rho}}(\mathcal{Q}_{-n}| \mathcal{P}_{-n})\, d\mu_{\rho}\\ &\leq &  \limsup_{n\to +\infty}  \frac{1}{n}  \int H_{\ov{\bf m}_{\rho}}(\mathcal{Q}_{-n}| \mathcal{P}_{-n})\, d\mu_{\rho}\: \leq\: \int H_{\ov{\bf m}_{\rho}}(\mathcal{Q}| \mathcal{P})\, d\mu_{\rho}.
\end{eqnarray*}
We  can group the tail elements in $\mathcal{Q}$  together with some care to  obtain a finite partition $\mathcal{P}$ satisfying the requirement in (\ref{Mane-2-inequ}).
\end{proof}

\section{The proof of Proposition \ref{mainprop}}

Let $\ov{\bf m}$ be as in Proposition \ref{mainprop}. To compare $h_{\ov{\bf m}}^s$ with   $h_{\rho_p}^s$, we first formulate  the  entropy  $\underline{h}_{\rho_{p}, \mathcal{P}}^s$  (see (\ref{underline-h-rho}))  in terms of some conditional entropy for the unconditional measure $\mu _\rho$.

Let $\W$ be a lamination of a compact metric space. A measurable partition  is said to be subordinated to $\W$ if its elements are bounded subsets of the leaves of $\W$ with non-empty interiors in the topology of the leaf. We can construct a  partition $\RR$ subordinated to $\ov{\W}^s$ by choosing a finite partition $\XX$ of $\OO^s(SM)$ into sufficiently small sets with non-empty interiors and subdivide each element of $\XX$ into the connected components of its intersection with the leaves.  We may assume  $\mathcal{R}$ is such that each element contains a ball with radius greater than some positive constant.  The partition $\RR$ is measurable if it is constructed as an intersection of an increasing family $\RR ^j , j \in \N,$ of finite partitions into measurable sets. 

Let $\PP$ be a finite partition of $\OO^s(SM)$ and we assume that we have chosen $\XX, \RR = \bigvee_j \RR^j $ as above and that $\PP $ refines $\XX$.  We may assume that the boundaries of the elements of $\PP, \XX $ and $\RR^j$ are all $\ov{\bf m}$-negligible.  The    conditional measures   $\ov {\bf m}^s_{\rho,u}$ in the  definition of $\underline{h}_{\rho, \mathcal{P}}^s$ can be taken on any measurable finite  partition $\RR$ chosen  in the above way,  so that 
\[\underline{h}_{\rho, \mathcal{P}}^s=\liminf_{n\to +\infty} -\frac{1}{n} \int \log \ov{\bf m}_{\rho,u}^s \left(\mathcal{P}_{-n}(\underline\vf, u)\right) \, d\mu _\rho (\underline \vf, u) = \liminf_{n\to +\infty} \frac{1}{n} H_{\mu_\rho} (\PP _{-n}| \RR).\]

Proving Proposition \ref{mainprop} amounts  to proving that, if $\rho_p\to -\infty$ and $ \ov{\bf m}_{\rho_p} \rightarrow \ov{\bf m}$ as $p \rightarrow +\infty $, then 
\begin{equation*}
h^s_{\ov{\bf m}} \; \geq \; \limsup _{p\to +\infty} \sup_\PP   \liminf_{n\to +\infty} \frac{1}{n} H_{\mu_{\rho_p}} (\PP _{-n}| \RR). \end{equation*} 
This is true, if we  can show,  for any $\alpha>0$,  there are partitions $\mathcal{P}, \mathcal{R}$ and $n$ large, such that for all $p$ large  enough, 
\begin{equation}\label{maineq-2}
h^s_{\ov{\bf m}}\geq \frac{1}{n} H_{\mu _{\rho _p}} (\PP _{-n}| \RR)  - 2\a\geq{ \underline{h}}^s_{\rho_p,\PP} - 3\a\geq  h^s_{\rho_p} -5\a. 
\end{equation}
The first inequality in  (\ref{maineq-2}) can be  achieved  if we can find good $\mathcal{P}, \mathcal{R}$  for $\ov{\bf m}$ with $h^s_{-\infty, \mathcal{P}}$ being close to $h^s_{\ov{\bf m}}$.  So we will show  the other two inequalities in (\ref{maineq-2}) first. 

We begin with  the  second inequality in (\ref{maineq-2}),  which  is not trivial in our setting   since the conditional entropy sequence $H_{\mu _{\rho _p}} (\PP _{-n}| \RR)$ is not necessarily  a subadditive sequence in $n$. 

\begin{lem} \label{subdivide} Given $\XX, \RR$ and $\PP$ as above, there exists a countable partition $\QQ$ of $\T$ such that the partition $\RR \bigvee \tau ^{-1} \PP \bigvee \tau ^{-1}\QQ $ is finer than $\tau^{-1} \RR$. Moreover, given $\a >0$,   there  are  $\d$ and $\L$ such that if the diameters of the elements of $\XX$ are smaller than $\d$ and if $\rho < \L,$ one can choose $\QQ$ with  $H_{\mu _\rho}(\QQ) < \a.$  
\end{lem}
\begin{proof}
For $u,w \in \OO^s(SM)$ in the same $\ov{\W}^s$ leaf, write $d^s(u,w)$ for the distance between $u$ and $w$ along their common leaf. For any $\d >0$,  there are  two  constants $c(\d)$ and  $C(\d) $ such that if $u$ and $w$ are on the same leaf and $d(u,w) <\d$, then either $d^s (u,w) < c (\d)$ or $d^s(u,w) \geq C(\d)$.  We  can ensure that $c(\d) \to 0$ as $\d \to 0$ and  that  $C(\d) \to +\infty$   as $\d \to 0.$    Suppose $u$ and $w$ are in the  same element of the partition $\RR$ and that $\vf_0 u $ and $\vf _0 w $ are in the same element of $\XX$. If $d^s (\vf _0 u, \vf _0 w ) < C(\d)$, in particular,   as soon as  $d^s(u,w) < C(\d) / \|\vf _0\|_{C^1}$, then $\vf_0 u $ and $\vf _0 w $ are in the same connected component of $\ov {\W}^s $ and thus in the same element of $\RR $. 

To obtain Lemma \ref{subdivide}, it is therefore enough to take the partition $\QQ$ of $\T$ as follows: the projection on $\SS $ depends only on the first coordinate $\vf_0$ and is the partition $A_n, n\geq 0,$ 
where $A_n := \{ \vf _0: nC(\d)  \leq \|\vf _0\|_{C^1} \leq (n+1) C(\d) \};$ $A_0 \times \OO^s(SM) $ is one element of $\QQ ;$ on each $A_n, n>0$, we further cut $\OO^sSM$ into $N_n $ pieces of diameter smaller than $1/(n+1)$.

The entropy of $\QQ$ satisfies 
\[ H_{\mu _\rho}(\QQ) \leq H_{\mu _\rho}\big(\{A_n:  n\geq 0\}\big) + {\bf c} \sum _{n=1}^\infty  \nu_\rho (A_n) \log n,\]
where ${\bf c}$ is some constant depending on the geometry of $\ov {\W}^s$.
Given $\a >0$, we will have $H_{\mu _\rho}(\QQ) < \a$ as soon as $\nu_\rho (\{ \vf : \|\vf \|_{C^1} > C(\d) \}) $ and the integral $\int _{\{ \vf : \|\vf \|_{C^1} > C(\d) \}} \log  \|\vf \|_{C^1} \, d\nu _\rho $ are sufficiently small. These two conditions can be realized by choosing $\d$ small and $\rho$ close enough to $-\infty$. 
\end{proof}

\begin{prop}\label{almostsub-add} Given $\a>0$, there is $\d >0 $ and $\L$ such that, for all $n>0$,  if the  diameter of the elements of $\XX$ are smaller than $\d$ and $ \rho < \L$, 
\begin{equation}\label{equ-sub-add} \frac{1}{n} H_{\mu_\rho} (\PP _{-n}| \RR) \; \geq \;    \liminf_{n\to +\infty} \frac{1}{n} H_{\mu_\rho} (\PP _{-n}| \RR) -\a=  \underline{h}^s_{\rho,\PP} - \a.\end{equation} \end{prop}

\begin{proof}
 Let $\mathcal{Q}$ be as in Lemma \ref{subdivide}.  Then we have that the mapping $n \mapsto H_{\mu _\rho} \big(\PP _{-n} \bigvee \QQ_{-n} \big| \RR\big) $ is subadditive.  Indeed,  for $n, n'\in \Bbb N$, \begin{eqnarray*}
 H_{\mu _\rho} \left(\PP _{-(n+n')} \bigvee \QQ_{-(n+n')} \big| \RR\right) &=& H_{\mu _\rho} \left(\PP _{-n} \bigvee \QQ_{-n} \big| \RR\right) \\
 && +\ H_{\mu _\rho} \left(\PP _{-(n+n')}^{-n} \bigvee \QQ_{-(n+n')}^{-n} \big| \RR \bigvee \PP _{-n} \bigvee \QQ_{-n}\right),
 \end{eqnarray*}
where $\PP _{-(n+n')}^{-n}:=\tau^{-n}\mathcal{P}\bigvee \cdots\bigvee\tau^{-(n+n'-1)}\mathcal{P}$ and $\QQ _{-(n+n')}^{-n}$ is defined in the same way.
Moreover, by Lemma \ref{subdivide}, the partition $\RR \bigvee \PP _{-n} \bigvee \QQ_{-n}$ is finer than $\tau ^{-n} \RR$ and the last term is smaller than $H_{\mu _\rho} \big(\PP _{-(n+n')}^{-n} \bigvee \QQ_{-(n+n')}^{-n} \big|\tau ^{-n} \RR \big).$  The desired  subaddivity follows by invariance of $\mu _\rho $ under $\tau ^n .$  Hence (\ref{equ-sub-add}) follows  since 
\begin{eqnarray*}
\liminf_{n\to +\infty}\frac{1}{n} H_{\mu _\rho} (\PP _{-n} | \RR) & \leq & \liminf_{n\to +\infty} \frac{1}{n} H_{\mu _\rho} \left(\PP _{-n} \bigvee \QQ_{-n} \big| \RR\right)\\
& =&\inf _n \frac{1}{n} H_{\mu _\rho} \left(\PP _{-n} \bigvee \QQ_{-n} \big| \RR\right)\\
& \leq & \inf_n \frac{1}{n} H_{\mu _\rho} (\PP _{-n}  | \RR) + H_{\mu _\rho}(\QQ) \\
& \leq & \inf_n \frac{1}{n} H_{\mu _\rho} (\PP _{-n}  | \RR) + \alpha. 
\end{eqnarray*}
\end{proof}

\begin{prop}\label{new-prop4.3} Let $\mathsf M\in \Bbb N$ and  let  $\PP$ be as in Proposition \ref{almostsub-add}. Then 
\begin{equation}\label{mathsfM-2}
\underline{h}_{\rho, \mathcal{P}}^{s, \mathsf M}\leq  {\mathsf M} \cdot \left({\underline{h}}^s_{\rho, \PP}+\alpha\right).
\end{equation}
\end{prop}
\begin{proof} Let $\mathcal{Q}$ be as in Lemma \ref{subdivide}.  Recall that \[\underline{h}_{\rho, \mathcal{P}}^{s, \mathsf M}=\liminf_{n\to +\infty} -\frac{1}{n} \int \log \ov{\bf m}_{\rho,u}^s \left(\mathcal{P}_{\mathsf M, -n}(\underline\vf, u)\right) \, d\mu _\rho (\underline \vf, u)=\liminf_{n\to +\infty} \frac{1}{n} H_{\mu_\rho} \big(\mathcal{P} _{\mathsf M, -n}\big| \RR\big).\] Hence, 
\begin{align*}
\underline{h}_{\rho, \mathcal{P}}^{s, \mathsf M}\leq \underline{h}_{\rho, \mathcal{P}_{-\mathsf M}}^{s, \mathsf M}=& {\mathsf M}\cdot \liminf_{n\to +\infty} \frac{1}{n{\mathsf M}} H_{\mu_\rho} \big(\mathcal{P}_{-\mathsf M n}\big| \RR\big)\notag\\
\leq &  {\mathsf M}\cdot \liminf_{n\to +\infty} \frac{1}{n{\mathsf M}} H_{\mu_\rho} \left(\left.\mathcal{P}_{-\mathsf M n}\bigvee \QQ_{-\mathsf M n} \right| \RR\right)\notag\\
=&  {\mathsf M}\cdot \liminf_{n\to +\infty} \frac{1}{n} H_{\mu_\rho} \left(\left.\mathcal{P}_{-n}\bigvee \QQ_{-n} \right| \RR\right)\notag\\
\leq & {\mathsf M} \cdot \left({\underline{h}}^s_{\rho, \PP}+H_{\mu _\rho}(\QQ)\right)\notag\\
\leq &  {\mathsf M} \cdot \left({\underline{h}}^s_{\rho, \PP}+\alpha\right). 
\end{align*}
\end{proof}

Next we show the last inequality in (\ref{maineq-2}).   For this, we  first state the results extending to our context the classical results of \cite{B}, \cite{Y} and \cite{Bu} (compare with \cite{CY}).

For $u \in \OO^s (SM), \un \vf \in \S, \eta >0$ and $ n \in \N$, define the \emph{random $\ov \W^s$-Bowen  ball} by
\[ B^s (\underline \vf,  u,  \eta, n )  := \left\{ {w} \in \ov W^s(u) :  \; d\left(\underline\vf|_k (w),  \underline\vf|_k (u)\right) < \eta  {\textrm { for } }  0 \leq k \leq n\right\}.\]
The following notion was introduced by Bowen (\cite{B}) for a single map and by  Cowieson-Young (\cite{CY}) in the random case. Since our mappings are smooth only along the 
foliation $\ov \W^s$, we introduce a variant by restricting to the leaves $\ov W^s.$
Fix $\zeta >0$ and a sequence $\un \vf \in \S$.  We denote  for $u\in \OO^s (SM), \, \eta >0$ and $n \in \N,$ $r(\zeta, \un \vf, u, \eta, n)$ the smallest number of  random $\ov \W^s$-Bowen balls $B^s(\un \vf, w, \eta, n)$ needed to cover the  random $\ov \W^s$-Bowen ball $B^s(\un \vf, u, \zeta, n)$. We then set  \[ h^s_{loc }(\zeta , \un \vf) \; := \; \sup_{u \in \OO^s(SM)} \lim\limits_{\eta \to 0} \limsup _{n\to +\infty} \frac{1}{n} \log r(\zeta, \un \vf, u, \eta, n).\]
The function $\un\vf \mapsto h^s_{loc }(\zeta , \un \vf) $ is $\s$-invariant; we denote $h^s_{loc ,\rho }(\zeta )$ its $\nu _\rho^{\otimes \N\cup \{0\}}$-a.e. value. 

 The following three propositions (Proposition \ref{Bowen}, Proposition \ref{killconstant}  and Proposition \ref{volumegrowth}) are proven in \cite{CY} for the global entropy with the additional hypotheses that $\nu_\rho$ are supported in a fixed neighborhood $\mathcal N$ of $\ov {\bf \Phi }_{-1}$ in $D^\infty (\OO^sSM) $ and that $\nu_\rho $ converge to $\nu _{-\infty}$ as $\rho \to - \infty$, in the sense that any  $D^\infty (\OO ^sSM) $ neighborhood of $\ov {\bf \Phi }_{-1}$  has eventually full measure for $\nu _\rho.$ In our case, we have two extensions of the argument in \cite{CY}: one is that the distributions $\nu _\rho $ are  not supported on a neighbourhood of $\ov {\bf \Phi }_{-1}$, but there is a tail; the other extension is that our mappings are not smooth everywhere, but only along the leaves of the foliation $\ov \W^s$.

\begin{prop}\label{Bowen} Given $\a >0, \zeta >0$, let $\XX$ be as in Proposition  \ref{almostsub-add}. Assume that the diameters of the elements of $\PP\cap \RR$ are all smaller than $\zeta .$ Then, for all $\rho$ close enough to $-\infty ,$
\begin{equation}\label{Bowen-S3} h^s_\rho - {\underline{h}}^s_{\rho, \PP} \; \leq \; h^s_{loc ,\rho }(\zeta)+\alpha.
\end{equation}
 \end{prop}
\begin{proof} 
Let $\mathsf M$ be a fixed positive integer   and  set 
$
h^{s, \mathsf M}_\rho := \sup_{\PP}\underline{h}_{\rho, \mathcal{P}}^{s, \mathsf M}$. 
Since
\[
\underline{h}_{\rho, \mathcal{P}_{-\mathsf M}}^{s, \mathsf M}={\mathsf M}\cdot \liminf_{n\to +\infty} \frac{1}{n{\mathsf M}} H_{\mu_\rho} \big(\mathcal{P}_{-\mathsf M n}\big| \RR\big)\geq {\mathsf M} \cdot {\underline{h}}^s_{\rho, \PP}, 
\]
we have 
\begin{equation}\label{mathsfM-1}
h^{s, \mathsf M}_\rho \geq   \sup_{\PP}\underline{h}_{\rho, \mathcal{P}_{-\mathsf M}}^{s, \mathsf M} \geq {\mathsf M}\cdot  \sup_{\PP}\underline{h}_{\rho, \mathcal{P}}^{s, \mathsf M}={\mathsf M}\cdot h^{s}_\rho. 
\end{equation}
Following \cite[Section 3]{B},  we obtain in our random setting that there is some positive constant ${\bf  c}$ which depends on the geometry of $\ov {\W}^s$ such that  for any $\beta>0$, 
\[
h^{s, \mathsf M}_\rho \leq \underline{h}_{\rho, \mathcal{P}}^{s, \mathsf M}+{\mathsf M}\big(h^s_{loc ,\rho }(\zeta)+\beta\big)+{\bf c}. 
\]
Using  (\ref{mathsfM-2}) and (\ref{mathsfM-1}), we deduce that \[
h^{s}_\rho\leq \underline{h}_{\rho, \mathcal{P}}^{s}+h^s_{loc ,\rho }(\zeta)+\alpha+\beta+\frac{1}{\mathsf{M}}{\bf  c}. 
\]
Letting $\beta\to 0$ and then $\mathsf M\to +\infty$, we obtain the inequality (\ref{Bowen-S3}). \end{proof}

Let $\mathsf M$ be a fixed positive integer. We define  for $u \in \OO^s (SM), \un \vf \in \S, \eta >0$ and $ n \in \N$,
\[ B^{s,\mathsf M} (\underline \vf,  u,  \eta, n )  := \left\{ {w} \in \ov W^s(u) :  \; d\left(\underline\vf|_{k\mathsf M} (w),  \underline\vf|_{k\mathsf M} (u)\right) < \eta  {\textrm { for } }  0 \leq k \leq n\right\},\]
$r^\mathsf M (\zeta, \un \vf, u, \eta, n)$ the smallest number of  $B^{s,\mathsf M}(\un \vf, w, \eta, n)$ balls needed to cover the  $B^{s,\mathsf M} (\un \vf, u, \zeta, n)$ ball,
 \[ h^{s,\mathsf M} _{loc }(\zeta , \un \vf) \; := \; \sup_{u \in \OO^s(SM)} \lim\limits_{\eta \to 0} \limsup _{n\to +\infty} \frac{1}{n} \log r^{\mathsf M} (\zeta, \un \vf, u, \eta, n)\]
and $h^{s,\mathsf M} _{loc, \rho}(\zeta) $ the  $\nu _\rho^{\otimes \N\cup \{0\}}$-a.e. value of $h^{s,\mathsf M} _{loc }(\zeta , \un \vf) .$

\begin{prop}\label{killconstant} With the above notations,  we have, for all $\rho <0, \zeta >0,$\begin{equation}\label{mathsfM-3} h^s_{loc ,\rho }(\zeta ) \leq \frac{1}{\mathsf M} h^{s,\mathsf M} _{loc, \rho}(\zeta) .\end{equation} \end{prop}
 \begin{proof} Observe that   $B^s(\un \vf, u, \zeta, n\mathsf M)$ is a subset of 
$B^{s,\mathsf M} (\underline \vf,  u,  \zeta, n ) $, so we are going to cover  $B^{s,\mathsf M} (\underline \vf,  u,  \zeta, n ) $ with $B^s(\un \vf, w,\eta, n\mathsf M)$ balls, $\eta $ arbitrarily small. Start with a cover of $B^{s,\mathsf M} (\underline \vf,  u,  \zeta, n ) $ with  $B^{s,\mathsf M}(\un \vf, w_\ell, \eta, n)$ balls with $ 1 \leq \ell \leq r^\mathsf M (\zeta, \un \vf, u, \eta, n)$ and fix $K>0 $ big. 
Let $\varkappa (\un \vf) := \max \{ \| \vf |_k \|_{C^1} :\  0\leq k < \mathsf M\}$. If $\varkappa (\sigma ^{j\mathsf M} \un \vf) \leq K $ for all $j,$ $ 0\leq j < n,$ then  each $B^{s,\mathsf M}(\un \vf, w_\ell, \eta, n)$ ball is contained in $B^s(\un \vf,  w_\ell , 2K \eta, n\mathsf M)$ and we take these  $B^s(\un \vf, w_\ell ,2 K \eta, n\mathsf M)$ balls  in our cover of  $B^{s,\mathsf M} (\underline \vf,  u,  \zeta, n )$. Otherwise, assume, for instance, that  $\varkappa ( \un \vf) > K$, we find,  for each $w_\ell ,$ at most ${\bf c} \lceil  \varkappa (\un \vf)/K \rceil^{2\ov m}$  points $w'_{\ell '}$ such that the union of the $B^s(\un \vf, w'_{\ell'}, 2K \eta, \mathsf M)$ balls cover $B^{s,\mathsf M}(\un \vf, w_\ell , \eta, 1)$, 
 where ${\bf c}$  is some positive  constant depending on the geometry of $\ov \W^s$  and $\lceil a \rceil$ denotes the smallest integer greater than $a$.  
Working inductively, we see that 
\[  r (\zeta, \un \vf, u, 2K \eta, n\mathsf M)\;\leq \; r^\mathsf M (\zeta, \un \vf, u, \eta, n) \;\Pi _{j=0}^{n -1} \lceil \varkappa (\sigma ^{j\mathsf M} \un \vf) / K \rceil ^{\ov m} \;{\bf c}^{ \Pi _{j=0}^{n -1}\chi _{ \{\varkappa (\sigma ^{j\mathsf M} \un \vf) > K\}}}. \]
It follows that for all $K >0, \un \vf  \in \S,$
\[  \mathsf M h^{s} _{loc}(\zeta , \un \vf)  \leq  h^{s,\mathsf M} _{loc}(\zeta , \un \vf) + \limsup _{n\to +\infty} \frac{{\ov m}}{n} \sum_{j=o}^{n-1}\! \log \lceil \varkappa (\sigma ^{j\mathsf M} \un \vf) / K \rceil +  \limsup _{n\to +\infty} \frac{\log {\bf c}}{n} \sum_{j=o}^{n-1}\! \chi _{\{\varkappa (\sigma ^{j\mathsf M} \un \vf) > K\}}. \]
Finally, we get, for all $\rho <0$, all $\zeta >0, K>0,$  
\[   \mathsf M h^{s} _{loc, \rho }(\zeta ) \; \leq \;  h^{s,\mathsf M} _{loc, \rho}(\zeta , \un \vf) + {\ov m} \E \left[\log \lceil \varkappa (\un \vf) / K \rceil \right] + \P \left[\varkappa (\un \vf) > K\right] \log {\bf c} .\] 
 Since $ \E[ \log \varkappa ] < + \infty ,$ Proposition \ref{killconstant} follows by letting $K$ go to infinity.
\end{proof}

\begin{prop}\label{volumegrowth} Fix $\zeta >0 $ small and $\rho < 0$. For all $r \in \N,$ there is a positive constant $C(r)$ such that, for all $\mathsf M\in \Bbb N$, 
\begin{align*}   &h^{s,\mathsf M}_{loc ,\rho }(\zeta ) \\ &\  \leq \frac{{\ov m}}{r}\int \log \big( \max \big\{ \zeta ^{s-1} \big\| (\vf |_\mathsf M)|_{\ov W^s (u)} \big\|_{C^s} :\  1 \leq s \leq r , u \in \OO^s (SM)\big\}\big) \, d\nu _\rho^{\otimes \mathsf M} (\vf|_\mathsf M)  + \log  C(r).\end{align*} \end{prop}
\begin{proof}
Fix $r>0, \mathsf M   \in \N$, a sequence $\underline \vf \in \S $ and $\zeta >0$. Two points  ${w}, w' \in \ov W^s(u)$ are said to be  $(\mathsf M, n, \eta)$-separated if 
\[
\max\left\{d\left(\underline\vf|_{k\mathsf M} (w),  \underline\vf|_{k\mathsf M} (w')\right):\  0 \leq k \leq n\right\}>\eta. 
\]
It is clear that  $r^\mathsf M (\zeta, \un \vf, u, \eta, n)$ is bounded from above by $s^\mathsf M (\zeta, \un \vf, u, \eta, n)$, the maximal cardinality of a set of  $(\mathsf M, n, \eta)$-separated points in  $B^{s,\mathsf M} (\un \vf, u, \zeta, n)$.  Consider  the mappings  $\vf'_k = \underline{\vf}|_\mathsf M \circ \sigma ^{k\mathsf M}$ 
and  their standard magnifications $\widehat {\vf'}_k : B(0,2)^{\ov m} \to \R^{\ov m} $ as explained in \cite{CY}, page 1129. In particular, we have  $\|\widehat {\vf'}_k \|_{C^s} \leq \zeta^{s-1} \| {\vf'}_k \|_{C^s}.$  Using this, we  can estimate  $s^\mathsf M (\zeta, \un \vf, u, \eta, n)$ by following  almost verbatim  the argument for the proof of Proposition 3 in \cite{CY} (which is based on the `Renormalization' Theorem in \cite{Yo} and a telescoping construction in \cite{Bu}) and obtain some constant $C_1(r,{\ov m},{\ov m}) =: C(r)$ as in \cite[Theorem 4]{CY}  such that 
\begin{align*}
&s^{\mathsf M} (\zeta, \un \vf, u, \eta, n) \\
&\ \ \ \ \  \leq C(r)^n \left(\frac{4}{\eta} \right)^{\ov m} \prod_{k=0}^{n-1}  \left(\max \left\{  \zeta ^{s-1}\big\| (\vf '_k)|_{\ov W^s (u)} \big\|_{C^s} :\ 1 \leq s \leq r , u \in \OO^s (SM)\right\}\right)^{{\ov m}/r}.
\end{align*}
Since $\vf'_k$ are independent, the ergodic theorem gives Proposition \ref{volumegrowth}.
 \end{proof}

\begin{cor}\label{yomdin}For any $\a >0$, there exists $\zeta _0>0$ such that if $\zeta \leq \zeta_0$, then \[\limsup_{\rho \to  -\infty } h^s_{loc ,\rho }(\zeta )< \a.\]\end{cor}
\begin{proof}
Fix $r\geq 2.$  We choose $\mathsf M\in \Bbb N$ large such that $ \displaystyle \frac{1}{\mathsf M} \log C(r) \leq \frac{{\ov m}}{r} B_{1,1}$, where $B_{1,1} $ is  defined in (\ref{r-derivatives}). 
Fix $\zeta \leq 1, \rho < 0. $  By (\ref{mathsfM-3}), $h^s_{loc ,\rho }(\zeta ) \leq \frac{1}{\mathsf M} h^{s,\mathsf M} _{loc, \rho}(\zeta) .$ Therefore, 
\[h^s_{loc ,\rho }(\zeta ) \leq  \frac{{\ov m}}{r} B_{1,1} 
  + \frac{{\ov m}}{r\mathsf M} \int \log \!\left(\!\max_u  \left(  \big\| \vf |_\mathsf M|_{\ov W^s (u)} \big\|_{C^1}  +  \zeta  \sum _{2 \leq s \leq r}  \big\| \vf |_\mathsf M|_{\ov W^s (u)} \big\|_{C^s}  \right)\! \right)\, d\nu ^{\otimes \mathsf M} _\rho (\vf|_\mathsf M). \]
Write, for $\a>0$,  $\log ^+ \a := \max \{ \log \a, 0\}.$ We have, using $\log (\a _1 + \a_2 ) \leq  \log^+ \a_1 + \a _2,$  for $\alpha_1, \alpha_2>0$,
\begin{eqnarray*}  &  &\log \left(\max_u  \left(  \big\| \vf |_\mathsf M|_{\ov W^s (u)} \big\|_{C^1}  +  \zeta  \sum _{2 \leq s \leq r}  \big\| \vf |_\mathsf M|_{\ov W^s (u)} \big\|_{C^s}  \right) \right) \\ 
&&\ \leq\,  \log^+ \left(\max_u  \left(  \big\| \vf |_\mathsf M|_{\ov W^s (u)} \big\|_{C^1} \right) \right)+  \zeta  \sum _{2 \leq s \leq r} \max_u \left( \big\| \vf |_\mathsf M|_{\ov W^s (u)} \big\|_{C^s}  \right) \\
&&\ \leq\,  \sum_{k=0}^{\mathsf M -1} \log^+ \left(\max_u \big( \big\| \vf|_{\ov W^s (u)} \big\|_{C^1}\big) \circ \sigma ^k\right) + \zeta  \sum _{2 \leq s \leq r} \max_u \left( \big\| \vf |_\mathsf M|_{\ov W^s (u)} \big\|_{C^s}\right) .
\end{eqnarray*}
We get by integrating with respect to $\nu ^{\otimes \mathsf M} _\rho ,$
\[h^s_{loc ,\rho }(\zeta ) \leq  \frac{{\ov m}}{r} B_{1,1} +  \frac{{\ov m}}{r} B_{1,1}(\rho) + \zeta \frac{{\ov m}}{r\mathsf M} \sum _{2 \leq s \leq r}B_{s,\mathsf M}(\rho),\]
where $B_{1,1}(\rho),\ B_{2,\mathsf M}(\rho), \cdots,  B_{r,\mathsf M}(\rho)$ are  defined in (\ref{r-derivatives}).  Note that,   by  Proposition \ref{con-geo-flow} ii), 
\[
B_{s,\mathsf M} =\limsup_{\rho \to  -\infty } B_{s,\mathsf M}(\rho) \; < \; + \infty, \ \forall 1\leq s\leq r, 
\]
 hence 
\begin{align*}
\inf_{\zeta >0} \limsup_{\rho \to  -\infty } h^s_{loc ,\rho }(\zeta )\leq  \frac{2{\ov m}}{r} \inf_{\zeta >0} \left( B_{1,1} + \zeta  \frac{1}{2\mathsf M} \sum_{s=2}^{r} B_{s,\mathsf M} \right) = \frac{2{\ov m}}{r} B_{1,1}.
\end{align*}
Since $r$ is arbitrary,  the corollary follows. \end{proof}

\begin{proof}[Proof of Proposition \ref{mainprop}]
Fix  $\a >0. $ We can choose the diameters of the elements of $\XX$  smaller than $c\zeta_0$, where $c$ is a constant depending on the local geometry of the leaves so that the diameter of the elements of $\PP \cap \RR$ are smaller than $\zeta_0$ and Corollary \ref{yomdin} applies.   We can also ensure  that these diameters are smaller than $\d$ given by Proposition \ref{almostsub-add}.   We may assume that the boundaries of the elements of $\PP, \XX $ and $\RR^j$ are all $\ov{\bf m}$-negligible. 

By definition,  $ h_{\ov{\bf m}}^s \geq \liminf_{n\to +\infty} \inf_j \frac{1}{n} H_{\ov{\bf m}} (\PP _{-n}| \RR^j ).$   We can choose $n $ and $j$ so that 
\begin{equation}\label{mproof-eq-1} h^s_{\ov{\bf m}} \geq  \frac{1}{n} H_{\ov{\bf m}} (\PP _{-n}| \RR^j ) - \a.
\end{equation}
Consider now $\rho_p, p\in \N$,   such that  $\rho_p\to -\infty$ and $\ov{\bf m}_{\rho _p } \rightarrow \ov{\bf m}$ as $p\to +\infty$. For  $\P$-a.e.   $\om \in \Om,$ each  element of  the partition $\bigcap_{k=0}^n (\underline \vf _{\rho_p}|_k(\om))^{-1}\PP$ converge in the Hausdorff metric towards the corresponding element $\bigcap _{k=0}^n \ov {\bf \Phi }_k \PP$.  Note that  all these elements of $\PP_{-n}$, and the elements of $\RR^j$ have $\ov{\bf m} $ negligible boundaries. It follows that there exists $P \in \N$ such that for $p \geq P,$
\begin{equation}\label{mproof-eq-2} \frac{1}{n} H_{\ov{\bf m}} (\PP _{-n}| \RR^j )  \; \geq \; \frac{1}{n} H_{\mu _{\rho _p}} (\PP _{-n}| \RR^j )  - \a  \; \geq  \;\frac{1}{n} H_{\mu _{\rho _p}} (\PP _{-n}| \RR)  - \a .
\end{equation}
The second inequality holds because the partition $\RR$ is finer than $\RR^j$. By Proposition \ref{almostsub-add}, we have, by our choice of $\d$ and as soon as $\rho_p < \L,$
\begin{equation}\label{mproof-eq-3} \frac{1}{n} H_{\mu _{\rho _p}} (\PP _{-n}| \RR) \; \geq \; \underline{h}^s_{\rho_p,\PP} - \a \; \geq  \; h^s_{\rho_p} -2\a - h^s_{loc ,\rho_p }(\zeta ),
\end{equation}
where the second equality follows from Proposition \ref{Bowen}. Finally, using all the above  inequalities  (i.e., (\ref{mproof-eq-1}),  (\ref{mproof-eq-2}) and (\ref{mproof-eq-3})) and Corollary \ref{yomdin}, we  find that 
\[ h_{\ov{\bf m}}^s \; \geq \; \limsup _{p\to +\infty}  h^s_{\rho_p}  - 5\a .\]
Proposition \ref{mainprop} follows from the arbitrariness of $\a.$
\end{proof}


\begin{thebibliography}{amsalpha}

\bibitem[\bf Bow72]{B} R. Bowen, Entropy-expansive maps,  \emph{Trans. Amer. Math. Soc.}  {\bf 164} (1972), 323--331.

\bibitem[\bf BB95]{BB} J. Bahnm\"{u}ller and T. Bogensch\"{u}tz,  A Margulis-Ruelle inequality for random dynamical systems,  \emph{Arch. Math.}  {\bf 64} (1995), no. 3,  246--253. 




\bibitem[\bf BR75]{BR}  R. Bowen and D. Ruelle, The ergodic theory of Axiom A flows,  \emph{Invent. Math. } {\bf 29} (1975), no. 3, 181--202. 


\bibitem[\bf BY19]{BY}  A. Blumenthal and L.-S. Young, Equivalence of physical and SRB measures
in random dynamical systems, \emph{Nonlinearity} {\bf  32}  (2019), no. 4, 1494--1524. 


\bibitem[\bf Buz97]{Bu} J. Buzzi, Intrinsic ergodicity of smooth interval maps, \emph{Israel J. Math.} {\bf 100} (1997), 125--161. 



\bibitem[\bf CE86]{CE}  A. P. Carverhill  and K. D.  Elworthy,  Lyapunov exponents for a stochastic analogue of the geodesic flow,  \emph{Trans. Amer. Math. Soc.}  {\bf 295} (1986), no. 1,  85--105. 



\bibitem[\bf CY05]{CY} W. Cowieson and L.-S. Young, SRB measures as zero-noise limits, \emph{Ergod. Th. \& Dynam. Sys.} {\bf 25} (2005), no. 4,  1115--1138.


\bibitem[\bf EO73]{EO} P. Eberlein and B. O'Neill, {Visibility manifolds}, \emph{Pacific J. Math.}  {\bf 46} (1973), 45--109.


\bibitem[\bf Elw82]{El} K. D. Elworthy, \emph{Stochastic Differential Equations on Manifolds},  London Mathematical Society Lecture Note Series, 70. Cambridge University Press, Cambridge-New York, 1982. 


\bibitem[\bf Gar83]{Ga} L. Garnett, Foliations, the ergodic theorem and Brownian motion, \emph{J. Funct.
Anal.} {\bf 51} (1983), no. 3,  285--311. 




\bibitem[\bf Ham97]{H2}U. Hamenst\"{a}dt, Harmonic measures for compact
negatively curved manifolds, \emph{Acta Math.}  {\bf 178} (1997), no. 1, 
39--107. 


\bibitem[\bf Kai88]{Kai2}  V. A. Kaimanovich, Brownian motion on foliations: entropy, invariant measures, mixing, \emph{Funct. Anal. Appl.}  {\bf 22} (1988), no. 4, 326--328. 

\bibitem[\bf Kai90]{Kai} V. A. Kaimanovich, Invariant measures of the geodesic flow and
measures at infinity on negatively curved manifolds, Hyperbolic behaviour of dynamical systems (Paris, 1990). {\em Ann. Inst. H. Poincar\'e, Phys.  Th\'eor.}  {\bf 53} (1990), no. 4, 
361--393. 

\bibitem[\bf Kat82]{K} A. Katok, Entropy and closed geodesics,  \emph{Ergod. Th. \& Dynam. Sys.}  {\bf 2} (1982), no. 3--4,  339--365.  


\bibitem[\bf Kif74]{Ki} Y. Kifer,  Small random perturbations of certain  smooth dynamical systems,  \emph{Izv. Akad. Nauk. SSSR; Ser. Mat.} {\bf 38} (1974), 1091--1115. 

\bibitem[\bf Kif86]{Ki2} Y. Kifer, \emph{Ergodic Theory of Random Transformations}, Birkh\"auser, Basel (1986). 

\bibitem[\bf KY88]{KY} Y. Kifer and Y. Yomdin, Volume growth and topological entropy for random transformations, \emph{in} Dynamical Systems (College Park, MD, 1986--1987), \emph{Springer Lecture Notes Math.}  {\bf 1342} (1988), 361--373. 

\bibitem[\bf Kun90]{Ku2} H. Kunita, \emph{Stochastic flows and stochastic differential equations},  Cambridge university press, 1990. 



\bibitem[\bf Led95]{L5} F. Ledrappier,  Applications of dynamics to compact manifolds of negative curvature, \emph{Proceedings of the International Congress of Mathematicians,}  Vol. 1, 2 (Z\"{u}rich, 1994), 1195--1202, Birkh\"{a}user, Basel, 1995. 



\bibitem[\bf LS]{LS} F. Ledrappier and L. Shu,   Entropies for negatively curved manifolds, \emph{arXiv preprint arXiv:1905.02747}.

\bibitem[\bf  LQ95]{LQ} P.-D. Liu and  M. Qian,  \emph{Smooth ergodic theory of random dynamical systems, } Springer Lecture Notes in Math. {\bf 1606},   Berlin, 1995. 

\bibitem[\bf LS11]{LiS11} P.-D. Liu and L. Shu,  Absolute continuity of hyperbolic invariant measures for endomorphisms, \emph{Nonlinearity} {\bf 24} (2011), no. 5,  1596--1611. 




\bibitem[\bf Man79]{Man} A. Manning, 
Topological entropy for geodesic flows, \emph{
Ann. of Math.}  (2) {\bf 110}  (1979), no. 3,  567--573. 

\bibitem[\bf Man83]{Mane} R. Ma\~{n}\'{e},  A proof of Pesin's formula, \emph{Ergod. Th. \& Dynam. Sys.} {\bf 1} (1981),  no. 1,  95--102.  Errata to: ``A proof of Pesin's formula''  \emph{Ergod. Th. \& Dynam. Sys.} {\bf 3} (1983), no. 1,  159--160. 

 
 
 \bibitem[\bf Ose68]{Os} V. I. Oseledec,
A multiplicative ergodic theorem. Characteristic Ljapunov, exponents of dynamical systems, 
\emph{Trudy Moskov. Mat. Ob\v{s}\v{c}.} {\bf19} (1968),  179--210. 


\bibitem[\bf Rue96]{R} D. Ruelle,   Positivity of entropy production in nonequilibrium statistical mechanics,  \emph{J. Stat. Phys.}  {\bf 85} (1996), no. 1--2, 1--23. 



\bibitem[\bf Shu87]{SFL} M. Shub, \emph{Global stability of dynamical systems,}  Springer-Verlag, 1987, With the
collaboration of Albert Fathi and R\'{e}mi Langevin, Translated from the French by
Joseph Christy. 



\bibitem[\bf Sul83]{Su} D. Sullivan, The Dirichlet problem at infinity for a negatively curved manifold, \emph{J. Differential Geometry}  {\bf 18} (1983), no. 4,   723--732. 



\bibitem[\bf Thi92]{Th} P. Thieullen,  Fibres dynamiques. Entropie et dimension,  \emph{ Ann. Inst. H. Poincar\'{e} Anal. Non Lin\'{e}aire} {\bf 9} (1992), no. 2,  119--146. 



\bibitem[\bf Yue94]{Y} C. B. Yue,  Rigidity and dynamics around manifolds of negative curvature, 
\emph{Math. Res. Lett.}  {\bf 1} (1994),  no. 2, 123--147. 


\bibitem[\bf Yom87]{Yo} Y. Yomdin, Volume growth and entropy, \emph{Israel J. Math.}  {\bf 57} (1987), no. 3,  285--300. 
$C^k$-resolution of semialgebraic mappings. Addendum to: ``Volume growth and entropy'', \emph{Israel J. Math.}  {\bf 57} (1987), no. 3, 301--317.

\end{thebibliography}
\end{document}